\documentclass[12pt]{amsart}
\usepackage{amscd,amssymb, latexsym}
\usepackage{amsmath,amsthm,amstext, amsbsy, amssymb}
\usepackage{epsfig,graphics,graphicx,mathrsfs,float}
\usepackage{a4wide}
\DeclareMathOperator{\Arg}{Arg}
\DeclareMathOperator{\dist}{dist}
\DeclareMathOperator{\length}{length}
\DeclareMathOperator{\sign}{sgn}
\DeclareMathOperator{\Int}{Int}
\DeclareMathOperator{\AC}{\mathit{AC}}
\DeclareMathOperator{\Ell}{\mathcal E}
\DeclareMathOperator{\SE}{\mathcal S\mathcal E}
\DeclareMathOperator{\NSE}{\mathcal N\kern-1.5pt\mathcal S\mathcal E}

\newtheorem{theorem}{Theorem}

\newtheorem{lemma}{Lemma}
\newtheorem{corollary}{Corollary}
\newtheorem{proposition}{Proposition}
\newtheorem{definition}{Definition}
\newtheorem{conjecture}{Conjecture}
\newtheorem{problem}{Problem}
\newtheorem{atheorem}{Theorem}

\theoremstyle{remark}
\newtheorem{remark}{Remark}

\makeatletter
\@addtoreset{equation}{section}
\makeatother

\begin{document}

\title[On Dirichlet problem for second-order elliptic equations]{On Dirichlet
problem for second-order elliptic equations in the plane and uniform
approximation problems for solutions of such equations}

\author[A.~Bagapsh, K.~Fedorovskiy, M.~Mazalov]{Astamur Bagapsh,
Konstantin Fedorovskiy, Maksim Mazalov}

\begin{abstract}
We consider the Dirichlet problem for solutions to general second-order
homogeneous elliptic equations with constant complex coefficients. We prove
that any Jordan domain with $C^{1,\alpha}$-smooth boundary, $0<\alpha<1$, is
not regular with respect to the Dirichlet problem for any not strongly
elliptic equation $\mathcal Lf=0$ of this kind, which means that for any such
domain $G$ it always exists a continuous function on the boundary of $G$ that
can not be continuously extended to the domain under consideration to a
function satisfying the equation $\mathcal Lf=0$ therein. Since there exists
a Jordan domain with Lipschitz boundary that is regular with respect to the
Dirichlet problem for bianalytic functions, this result is near to be sharp.
We also consider several connections between Dirichlet problem for elliptic
equations under consideration and problems on uniform approximation by
polynomial solutions of such equations.
\end{abstract}

\thanks{The work presented in Sections~1--3 was partially supported by the
Program of the President of the Russian Federation for support young Russian
scientist (grant no.~MK-1204.2020.1), by the Ministry of Science and Higher
Education of the Russian Federation (theme No.~0705-2020-0047), and by the
Theoretical Physics and Mathematics Advancement Foundation ``BASIS''.\\
\indent The results of Section~5 were obtained within the frameworks of the
project 17-11-01064 by the Russian Science Foundation.}

\address{
\hskip -\parindent Astamur Bagapsh${}^{1,2,3}$:
\newline \indent ${}^1$)~Bauman Moscow State Technical University, Moscow 105005, Russia;
\newline \indent ${}^2$)~Federal Research Center `Computer Science and Control
\newline \indent \quad of the Russian Academy of Sciences, Moscow 119333, Russia;
\newline \indent ${}^3$)~Moscow Center for Fundamental and Applied Mathematics,
\newline \indent \quad Lomonosov Moscow State University, Moscow 119991, Russia.
\newline \indent {\tt a.bagapsh@gmail.com}
\smallskip
\newline \noindent Konstantin Fedorovskiy${}^{1,2,3}$:
\newline \indent ${}^1$)~Faculty of Mechanics and Mathematics \&
\newline \indent \quad Moscow Center for Fundamental and Applied Mathematics,
\newline \indent \quad Lomonosov Moscow State University, Moscow 119991, Russia;
\newline \indent ${}^2$)~Saint Petersburg State University, St.~Petersburg 199034, Russia;
\newline \indent ${}^3$)~Bauman Moscow State Technical University, Moscow 105005, Russia.
\newline \indent {\tt kfedorovs@yandex.ru}
\smallskip
\newline \noindent Maksim Mazalov${}^{1,2}$:
\newline \indent ${}^1$)~Smolensk Branch of the Moscow Power Engineering Institute,
\newline \indent \quad Smolensk 214013, Russia;
\newline \indent ${}^2$)~Saint Petersburg State University, St.~Petersburg 199034, Russia.
\newline \indent {\tt maksimmazalov@yandex.ru}
}

\maketitle

\section{Introduction and description of main result}
\label{s:intro}

Let $\mathcal L$ be a second-order elliptic homogeneous partial differential
operator in the complex plane $\mathbb C$ with constant \emph{complex}
coefficients. That is
\begin{equation}\label{eq-ellop}
\mathcal Lf=c_{11}\frac{\partial^2f}{\partial x^2}+
2c_{12}\frac{\partial^2f}{\partial x\partial y}+
c_{22}\frac{\partial^2f}{\partial y^2},
\end{equation}
where $c_{11},c_{12},c_{22}\in\mathbb C$. Throughout this paper, $z$ will
mean both a complex number $x+iy$ and the eponymous point $(x,y)$ in the
$2$-dimensional plane. As usual, $\overline{z}=x-iy$ stands for the complex
conjugate to $z$.

Recall that the ellipticity of $\mathcal L$ means that the expression
$c_{11}\xi_1^2+2c_{12}\xi_1\xi_2+c_{22}\xi_2^2$ (the \emph{symbol} of
$\mathcal L$) does not vanish for real $\xi_1$ and $\xi_2$ unless
$\xi_1=\xi_2=0$. It may be readily verified that the ellipticity of $\mathcal
L$ is equivalent to the property that both roots $\lambda_1$ and $\lambda_2$
of the corresponding \emph{characteristic} equation
$c_{11}\lambda^2+2c_{12}\lambda+c_{22}=0$ are not real. We denote by $\Ell$
the class of all elliptic operators of the form \eqref{eq-ellop}.

Let $U\subset\mathbb C$ be an open set. A complex-valued function $f$ is
called $\mathcal L$-analytic on $U$, if it is defined on $U$ and satisfies
there the equation
\begin{equation}\label{eq-elleq}
\mathcal Lf=0,
\end{equation}
which is treated in the classical sense. Denote by $\mathcal O(U,\mathcal L)$
the class of all $\mathcal L$-analytic functions on $U$. One ought to recall
that any continuous function $f$ on $U$ satisfying the
equation~\eqref{eq-elleq} in the sense of distributions is real-analytic in
$U$ and satisfies this equation in $U$ in the classical sense (see, for
instance, \cite{Tre1961book}, Theorem~18.1).

Let us highlight two most typical examples of operators under consideration.
Put for brevity $\partial_x=\partial\big/\partial x$ and
$\partial_y=\partial\big/\partial y$. The first example is the Laplace
operator $\Delta= (\partial_x)^2+(\partial_y)^2$. In this case $\lambda_1=i$
and $\lambda_2=-i$. The class $\mathcal O(U,\Delta)$ consists of all
(complex-valued) \emph{harmonic} functions in $U$, and every function
$f\in\mathcal O(U,\Delta)$ has the form $f(z)=h(z)+g(\overline{z})$, where
$h$ and $g$ are holomorphic functions in $U$ and $\{\overline{z}\colon z\in
U\}$, respectively.

The second example is the operator $\overline\partial^2$, where
$\overline\partial=\partial_{\overline{z}}=\frac12\big(\partial_x+i\partial_y\big)$
is the standard Cauchy--Riemann operator. The operator $\overline\partial^2$
is often called the Bitsadze operator. For $\mathcal L=\overline\partial^2$
we have $\lambda_1=\lambda_2=-i$. The functions $f\in\mathcal
O(U,\overline\partial^2)$ are called \emph{bianalytic} functions (in $U$),
and it is clear that every such function has the form
$f(z)=\overline{z}h(z)+g(z)$, where $h$ and $g$ are holomorphic functions in
$U$.

The possibility to express a given $\mathcal L$-analytic function by means of
a pair of holomorphic functions, noted in both examples given, remains true
for general $\mathcal L\in\Ell$, and this circumstance is one of keystones
for our further considerations and constructions.

In what follows $C(X)$ will stand for the space of all bounded and continuous
complex-valued functions on a closed set $X\subset\mathbb C$.

In this paper we are interested in the problem to find conditions for a
bounded domain $G$ in $\mathbb C$, which ensure that every function of class
$C(\partial G)$ can be continuously extended to a function being continuous
on $\overline{G}$ and $\mathcal L$-analytic in $G$. In other words, we are
dealing with the question whether a given domain $G$ is regular with respect
to the Dirichlet problem for $\mathcal L$-analytic functions which we will
call $\mathcal L$-Dirichlet problem for brevity.

This problem is by no means the only question associated with the Dirichlet
problem for the equation \eqref{eq-elleq}. There is a number of works that
deal with this problem in various classes of functions (for instance, in
$L^p$-spaces, in Sobolev spaces, etc.). Moreover, a plenty of works deal with
Dirichlet problem for elliptic equations and systems of equations with
varying coefficients of certain classes. In spite of the significant interest
and importance of these problems and results obtained we will not touch them
here. In our studies of $\mathcal L$-Dirichlet problem we are motivated by
problems on uniform approximation by $\mathcal L$-analytic polynomials, that
is by polynomial solutions of the equation \eqref{eq-elleq}. In this context
we need exactly the description of $\mathcal L$-regular domains. We turn now
to the exact formulation of the problem in question.

\begin{definition}
A bounded domain $G\subset\mathbb C$ is called regular with respect to the
$\mathcal L$-Dirichlet problem or, shortly, $\mathcal L$-regular, if for
every function $\psi\in C(\partial G)$ there exists a function $f_\psi\in
C(\overline{G})\cap\mathcal O(G,\mathcal L)$ such that $f_\psi|_{\partial
G}=\psi$.
\end{definition}

\begin{problem}\label{prob-Dirichlet}
Given $\mathcal L\in\Ell$, to find necessary and sufficient conditions for a
bounded domain $G\subset\mathbb C$ to be $\mathcal L$-regular.
\end{problem}

It turns out that Problem~\ref{prob-Dirichlet} differs significantly in the
following two mutually complementary cases: in the case when $\mathcal L$ is
strongly elliptic, and in the opposite one. Let us recall how these classes
of elliptic operators are defined.

\begin{definition}
An operator $\mathcal L$ of the type \eqref{eq-ellop} is \emph{strongly
elliptic}, if its characteristic roots $\lambda_1$ and $\lambda_2$ belong to
different half-planes with respect to the real line, that is, if
$\sign\Im\lambda_1\neq\sign\Im\lambda_2$. We denote by $\SE$ the class of all
strongly elliptic operators of the form \eqref{eq-ellop}; thus
$\SE=\{\mathcal L\in\Ell\colon \sign\Im\lambda_1\neq\sign\Im\lambda_2\}$, and
we put $\NSE=\Ell\setminus\SE$.
\end{definition}

Formally this definition of strong ellipticity differs from the classical one
due to Vishik \cite{Vis1951sbm}, but it can be derived from it in the case
under consideration.

For operators $\mathcal L\in\SE$ several results are obtained about $\mathcal
L$-regularity under certain restrictions on $\mathcal L$ and on the class of
domains under consideration. First of all, one ought to state the famous
result due to A.~Lebesgue \cite{Leb1907rcmp}, which sounds as follows:
\begin{atheorem}\label{thm-lebesgue}
Let $G$ be an arbitrary bounded simply connected domain in $\mathbb C$. Then
$G$ is $\Delta$-regular \textup(i.e. regular with respect to the standard
Dirichlet problem for harmonic functions\textup).
\end{atheorem}
This result is one of keystones, underlying the proof of the celebrated
Walsh--Lebesgue criterion for uniform approximation by harmonic polynomials
on compact sets in the complex plane. A brief account concerning the
corresponding topic in Approximation Theory and the role of
Problem~\ref{prob-Dirichlet} in this themes, will be presented in the final
section of this paper.

It is a clear that the result similar to Theorem~\ref{thm-lebesgue} also
takes place for every $\mathcal L\in\SE$ possessing the property
$\lambda_1=\overline\lambda_2$ (such operators $\mathcal L$ are exactly the
operators with real coefficients, up to a common complex multiplier).

To the best of our knowledge, the conditions of $\mathcal L$-regularity of
domains for general $\mathcal L\in\SE$ were obtained only under additional
fairly stringent constraints on the properties of $G$. For instance, the
following result was proved in \cite{VerVog1997tams}, Theorem~7.4:
\begin{atheorem}\label{thm-vervog}
Let $G$ be a Lipschitz domain whose boundary consists of a finite number of
$C^1$-curves. Then $G$ is $\mathcal L$-regular for any $\mathcal L\in\SE$.
\end{atheorem}
Without going into further details, we note that all known results about
$\mathcal L$-regularity of bounded simply connected domains in $\mathbb C$ in
the case of general operators $\mathcal L\in\SE$ are quite far from to cover
even the case of general Jordan domains.

In the case of operators which are not strongly elliptic,
Problem~\ref{prob-Dirichlet} remains quite poorly studied. The almost only
considered case is the one where $\mathcal L=\overline\partial^2$ (the square
of the Cauchy--Riemann operator). The $\overline\partial^2$-Dirichlet problem
was studied in several works, see, for instance, \cite{CarParFed2002sbm} and
\cite{Maz2009sbm}. The following results were obtained in \cite{Maz2009sbm},
Theorem~1 and Example~2:
\begin{atheorem}\label{thm-mazalov}\hfill\par
1.~Let $G$ be a Jordan domain with rectifiable boundary in $\mathbb C$, and
let $\varphi$ be some conformal mapping from $\mathbb D$ onto $G$. If
$\displaystyle\int_{\mathbb D}|\varphi''(z)|\,dxdy<\infty$, then $G$ is not
$\overline\partial^2$-regular.

\smallskip
2.~There exists a Jordan domain with Lipschitz boundary which is
$\overline\partial^2$-regular.
\end{atheorem}
It follows from this theorem, that Jordan domains with at least
$C^{1,\alpha}$-smooth boundaries, $0<\alpha<1$, are not
$\overline\partial^2$-regular; the exact definition of this class of domains
is given in Section~\ref{s:aux} below. Thus the situation in
Problem~\ref{prob-Dirichlet} for operators that are not strongly elliptic
looks ``turned upside down'' with respect to the strongly elliptic case:
domains with sufficiently smooth boundaries can not be regular, but some
special domains (having not too smooth boundaries) may have such behavior.

Problem~\ref{prob-Dirichlet} in the case of general $\mathcal L\in\NSE$ was
touched upon in \cite{Zai2006psim}, where it was proved that any Jordan
domain whose boundary contains some analytic arc is not $\mathcal L$-regular
for any $\mathcal L\in\NSE$ (see \cite{Zai2006psim}, Proposition~1).

In the present paper we consider Problem~\ref{prob-Dirichlet} for general
operators $\mathcal L\in\NSE$. Our main result --- Theorem~\ref{thm-dirprob}
stated in Section~\ref{s:aux} below --- asserts that \emph{Jordan domains
with $C^{1,\alpha}$-smooth boundary, $0<\alpha<1$, are not $\mathcal
L$-regular for such operators}. It is not clear at the moment whether this
result is sharp; but the part~2 of Theorem~\ref{thm-mazalov} shows that it is
``near to be sharp''. Although the example of a Jordan domain with the
boundary that is less regular than $C^{1,\alpha}$-smooth, which is however
$\mathcal L$-regular for some $\mathcal L\in\NSE$, is known only for
$\mathcal L=\overline\partial^2$, the general situation when domains with
sufficiently regular (smooth) boundaries are not $\mathcal L$-regular, while
domains having less regular boundaries may be $\mathcal L$-regular is rather
unexpected and essentially new. We also consider the problem on uniform
approximation by $\mathcal L$-analytic polynomials and its relations with
$\mathcal L$-Dirichlet problem and with weak maximum modulus principle for
$\mathcal L$-analytic functions.

The structure of the paper is as follows. In Section~\ref{s:aux} we present
the necessary background information. Also we formulate in this section one
result of a technical nature that underlies our proof of the main result.
Firstly we present this result in a somewhat informal form (see the estimate
\eqref{eq-Mn-estonsol}) and show how Theorem~\ref{thm-dirprob} can be derived
from it, and later on we provide an accurate formulation of this result, see
Theorem~\ref{thm-mainest}. The proof of Theorem~\ref{thm-mainest} is given in
Section~\ref{s:proof}. In Section~\ref{s:example} we give a schematic outline
of the construction given in \cite{Maz2009sbm} to verify the second statement
of Theorem~\ref{thm-mazalov}.

Finally, in Section~\ref{s:approx} we consider the problem about
approximation by $\mathcal L$-analytic polynomials and its connections with
Problem~\ref{prob-Dirichlet}. We present the new proof of the criterion for
uniform approximability of functions by $\mathcal L$-analytic polynomials on
boundaries of Carath\'eodory domains, see Theorem~\ref{thm-tz}. This result
was firstly obtained in \cite{Zai2002mn}, but the proof given there is rather
involved technically and, moreover, it is not enough complete in a certain
place. As a consequence of Theorem~\ref{thm-tz} one can show that weak
maximum modulus principle (i.e. a maximum modulus principle with a constant
depending on the domain under consideration) is certainly failed for any
$\mathcal L\in\NSE$.

Through the paper we will use the following common notations. For a given
closed set $X\subset\mathbb C$ the space $C(X)$ will be endowed with the
standard uniform norm $\|f\|_X=\sup_{z\in X}|f(z)|$. When $X=\mathbb C$ we
will write $\|f\|$ instead of $\|f\|_{\mathbb C}$. We will denote by $\mathbb
D$ and $\mathbb T$ the unit disk and the unit circle in $\mathbb C$, that is
$\mathbb D=\{z\colon |z|<1\}$ and $\mathbb T=\{z\colon |z|=1\}$. The symbol
$D(a,r)$ will stand for the open disk in $\mathbb C$ with center $a$ and
radius $r$, while $m_2(\cdot)$ will stand for the 2-dimensional Lebesgue
measure. Moreover, we will denote by $C,C_1,C_2,\ldots$ positive numbers
(constants) which are not necessarily the same in distinct formulae.


\section{Background and auxiliary results}
\label{s:aux}

\subsection*{Solutions to the equation \eqref{eq-elleq}}

Let $\mathcal L\in\Ell$. Let $\lambda_1,\lambda_2$ be the characteristic
roots of $\mathcal L$, that is $c_{11}\lambda_s^2+2c_{12}\lambda_s+c_{22}=0$,
$s=1,2$, and $\lambda_1,\lambda_2$ are not real. Then $\mathcal L$ may be
represented in the following form
\begin{equation}\label{eq-op1}
\mathcal L=\left\{\begin{array}{ll}
c_{11}\big(\partial_x-\lambda_1\partial_y\big)
\big(\partial_x-\lambda_2\partial_y\big),
& \text{if}\quad \lambda_1\neq\lambda_2,\\[1ex]
c_{11}\big(\partial_x-\lambda\partial_y\big)^2,
& \text{if}\quad \lambda_1=\lambda_2=\lambda.
\end{array}\right.
\end{equation}

Let $U$ be an open set in $\mathbb C$. Using \eqref{eq-op1} one can show
(see, for instance, \cite{ParFed1999sbm}, Proposition~2.1) that every
function $f\in\mathcal O(U,\mathcal L)$ may be expressed in terms of a pair
of holomorphic functions in the following form. When
$\lambda_1\neq\lambda_2$, the function $f$ has the form
\begin{equation}\label{eq-solrep1}
f(z)=f_1(T_{(1)}z)+f_2(T_{(2)}z),
\end{equation}
where $T_{(1)}z=x+\lambda_2^{-1}y$, $T_{(2)}z=x+\lambda_1^{-1}y$, and where
$f_1$ and $f_2$ are holomorphic functions in $\{T_{(1)}z\colon z\in U\}$ and
$\{T_{(2)}z\colon z\in U\}$, respectively. One ought to emphasize, that
$\lambda_1$ and $\lambda_2$ are not real. Next, if
$\lambda_1=\lambda_2=\lambda$, then $f$ has the form
\begin{equation}\label{eq-solrep2}
f(z)=(T_{(1)}z)f_1(T_{(2)}z)+f_0(T_{(2)}z),
\end{equation}
where $T_{(1)}z=x-\lambda^{-1}y$, $T_{(2)}z=x+\lambda^{-1}y$, and where $f_0$
and $f_1$ are holomorphic functions in $\{T_{(2)}z\colon z\in U\}$.

For example, if $\mathcal L=\Delta$, then $\lambda_1=i$, $\lambda_2=-i$, and
hence $T_{(1)}z=z$, $T_{(2)}z=\overline{z}$ and \eqref{eq-solrep1} is the
standard decomposition of a harmonic function onto sum of its holomorphic and
antiholomorphic parts. Similarly, for $\mathcal L=\overline\partial^2$ we
have $\lambda_1=\lambda_2=-i$, $T_{(1)}z=\overline{z}$, $T_{(2)}z=z$, and
\eqref{eq-solrep2} looks in this case as a polynomial on $\overline{z}$ of
degree $1$ with holomorphic coefficients, which is the standard form of a
generic bianalytic function.

In what follows we will work with slightly different representation of
$\mathcal L$ and, respectively, with different representation of $\mathcal
L$-analytic functions. It turns out that one can find a not degenerate
real-linear (that is linear over the reals) transformation of the plane that
reduces $\mathcal L$ to the form
\begin{equation}\label{eq-op2b}
\mathcal L_*=c\overline\partial\partial'_\beta,
\end{equation}
where $c\in\mathbb C$, $\beta\in\mathbb R$, $|\beta|\geqslant1$, and
$\partial'_\beta=\partial_x+i\beta\partial_y$. This representation was used
in several papers and it turned out to be quite useful (see, for instance,
\cite{Zai2004izv} and \cite{Zai2006psim}), but we need to modify it a bit
more to get a more simple notation system which allows one to distinguish
strongly elliptic and not strongly elliptic cases in a more clear way. Note
that $\beta\leqslant-1$ if and only if $\mathcal L\in\SE$, while
$\beta\geqslant1$ if and only if $\mathcal L\in\NSE$.

Given $\tau\in\mathbb C$ with $|\tau|<1$, we put
\begin{equation}\label{eq-dtau}
\partial_{\tau}=\overline\partial+\tau\partial,
\end{equation}
where $\partial=\partial_z=\frac12\big(\partial_x-i\partial_y\big)$.
Sometimes the operator $\partial$ is called the conjugate (or
antiholomorphic) Cauchy--Riemann operator.

Similarly to representation of $\mathcal L$ in the form \eqref{eq-op2b}, it
can be shown that $\mathcal L$ can be reduced by means of a suitable not
degenerate real-linear transformation of the plane to the form
\begin{equation}\label{eq-op2se}
\mathcal L_*=c\partial\partial_\tau,\quad
c=c(\mathcal L)\in\mathbb C,\quad \tau=\tau(\mathcal L)\in[0,1),
\end{equation}
when $\mathcal L\in\SE$, or
\begin{equation}\label{eq-op2nse}
\mathcal L_*=c\overline\partial\partial_\tau,\quad
c=c(\mathcal L)\in\mathbb C,\quad \tau=\tau(\mathcal L)\in[0,1),
\end{equation}
when $\mathcal L\in\NSE$, respectively. Observe, that
$\partial\partial_0=\partial\overline\partial=\frac14\Delta$, while
$\overline\partial\partial_0=\overline\partial^2$. Let us accent that
$0\leqslant\tau<1$ (i.e.$\tau$ is real) both in \eqref{eq-op2se} and in
\eqref{eq-op2nse}. In both these cases the characteristic root of
$\partial_\tau$ lies in the lower half-plane $\{\Im\lambda<0\}$, and the
operator $\partial_\tau$ itself is more ``close'' to $\overline\partial$ than
to $\partial$.

\begin{remark}\label{rem-solspace}
Let $\mathcal L$ be an arbitrary operator of the form \eqref{eq-ellop}, and
suppose $T_*$ to be the non degenerate real-linear transformation of the
plane that reduces $\mathcal L$ to the operator $\mathcal L_*$ of the form
\eqref{eq-op2b}, \eqref{eq-op2se} or \eqref{eq-op2nse}. It is not difficult
to show that $f\in\mathcal O(U,\mathcal L)$ if and only if $f\circ
T_*^{-1}\in\mathcal O(T_*U,\mathcal L_*)$, where $U$ is an open set in
$\mathbb C$.
\end{remark}

Therefore the question about $\mathcal L$-regularity of a given domain $G$ is
equivalent to the question about $\mathcal L_*$-regularity of the domain
$T_*G$. Bearing this in mind, we will always assume in what follows that the
operator under consideration is already given in the reduced form
\eqref{eq-op2se} or \eqref{eq-op2nse} with $c=1$. Let us clarify how the
solution representations \eqref{eq-solrep1} and \eqref{eq-solrep2} will look
in this case. Given $\tau$, $0\leqslant\tau<1$, and an open set
$U\subset\mathbb C$ we put
\begin{align}
\mathcal L_\tau&=\overline\partial\partial_\tau,\\
\mathcal L^\dag_\tau&=\partial\partial_\tau.
\end{align}
Therefore $\mathcal L_\tau\in\NSE$, while $\mathcal L^\dag_\tau\in\SE$.
Moreover, we put
$$
\mathcal O_\tau(U)=\mathcal O(U,\mathcal L_\tau),\qquad
\mathcal O^\dag_\tau(U)=\mathcal O(U,\mathcal L^\dag_\tau).
$$
Denote by $T$ the real-linear transformation of the plane defined by the
formula
$$
Tz=z_\tau,
$$
where
\begin{equation}\label{eq-ztau}
z_\tau=z-\tau\overline{z}.
\end{equation}
Since $\tau\in[0,1)$, then $T$ is a sense-preserving mapping (notice that the
Jacobian of $T$ is $1-\tau^2$). Moreover, it can be easily verified that
$\overline\partial z=0$, $\overline\partial z_\tau=-\tau$, $\partial_\tau
z=\tau$, $\partial_\tau z_\tau=0$.

It follows from \eqref{eq-solrep1} that any function $f\in\mathcal
O^\dag_\tau(U)$ has the form
\begin{equation}\label{eq-soltau-se}
f(z)=h(z_\tau)+g(\overline{z}),
\end{equation}
where $g$ and $h$ are holomorphic functions on $\{\overline{z}\colon z\in
U\}$ and $TU$, respectively. Next, if $\tau>0$, then any function
$f\in\mathcal O_\tau(U)$ has the form
\begin{equation}\label{eq-soltau-nse}
f(z)=h(z_\tau)+g(z),
\end{equation}
where $g$ and $h$ are holomorphic functions on $U$ and $TU$, respectively.
The remaining class $\mathcal O_0(U)=\mathcal O(U,\overline\partial^2)$
consists of bianalytic functions, and any function $f\in \mathcal O_0(U)$ has
the form $\overline{z}f_1(z)+f_0(z)$ where $f_0$ and $f_1$ are holomorphic
functions in $U$.

Dealing with the case of not strongly elliptic equations, we assume that
$\mathcal L=\mathcal L_\tau$ for some $\tau\in[0,1)$. As it was mentioned
above, the problem we are interested in was studied in this case mainly for
bianalytic functions, while the general case remained quite poorly studied.
Note that the space $\mathcal O_0(U)$ has an additional algebraic structure,
in contrast to the space $\mathcal O_\tau(U)$ for $\tau>0$. Indeed, $\mathcal
O_0(U)$ is a module over the space of holomorphic functions on $U$ generated
by the function $\overline{z}$. This circumstance is one plausible reason
that explains new significant difficulties for working with functions of
class $\mathcal O_\tau(U)$, because many ideas and constructions which are
useful for bianalytic functions do not work properly for functions from
$\mathcal O_\tau(U)$, $\tau>0$. One ought to emphasize also that the class
$\mathcal O_\tau(U)$, $\tau>0$, is neither conformally invariant nor, even,
M\"obius invariant. It also causes additional difficulties for working with
this class. Moreover, we need to make the following observation.

\begin{remark}\label{rem-change}
Let $a\in\mathbb C$, $a\neq0$, and $b\in\mathbb C$. It can be readily
verified that $\mathcal L_\tau f(az+b)=|a|^{-1}\mathcal L_{\tau'} f(z)$,
where $\tau'=\tau\overline{a}\big/a$ (recall, that we have allowed complex
values of $\tau$ in the initial definition of $\partial_\tau$). Therefore,
the equation $\mathcal L_\tau f=0$ is invariant under shifts and dilations of
the plane, but this equations is changed under rotations of the plane as
follows: the rotation of the plane to the angle $\alpha$ leads to the
rotation of the parameter $\tau$ to the angle $2\alpha$ in the opposite
direction.
\end{remark}

The next lemma shows how functions from the space
$C(\overline{G})\cap\mathcal O_\tau(G)$ behave near the boundary of a given
domain $G\subset\mathbb C$. In this connection see also
\cite{BagFed2017psim}, Lemma~1, where one close result was proved in a
different manner.

\begin{lemma}\label{lem-solest}
Let $G$ be a bounded simply connected domain in $\mathbb C$, let
$\tau\in(0,1)$, and let $f\in C(\overline{G})\cap\mathcal O_\tau(G)$. For a
given point $a\in G$ take a point $a'\in\partial G$ such that
$|a-a'|=\dist(a,\partial G)$, and put $d=|a-a'|$ and
$d_\tau=|Ta-Ta'|=|a_\tau-a'_\tau|$, where the mapping $T\colon z\mapsto
z_\tau$ is defined by the formula \eqref{eq-ztau}. Then for every integer
$m\geqslant1$ the functions $g$ and $h$ from the representation
\eqref{eq-soltau-nse} for $f$ admit the estimates
\begin{align}
|h^{(m)}(a_\tau)|&\leqslant
C_1\frac{m!}{d_\tau^m}\omega(f,d_\tau),\label{eq-estsol-h}\\
|g^{(m)}(a)|&\leqslant
C_2\frac{m!}{d^m}\omega(f,d),\label{eq-estsol-g}
\end{align}
where $\omega(f,\cdot)$ stands for the modulus of continuity of $f$ on
$\overline{G}$.
\end{lemma}

\begin{proof}
It is enough to prove \eqref{eq-estsol-h}, the proof of the remaining
estimate \eqref{eq-estsol-g} is similar. Take an arbitrary $r<d_\tau$. For
every $z\in D(a,r)$ the following Taylor-type expansion holds
$$
h'(z_\tau)=\sum_{k=0}^{\infty}\frac{h^{(k+1)}(a_\tau)}{k!}(z_\tau-a_\tau)^k.
$$
Multiplying this decomposition by
$(\overline{z}_\tau-\overline{a}_\tau)^{m-1}$ and integrating thereafter over
the ellipse $D_\tau(a,r)=\{z\colon |(z-a)_\tau|<r\}$, where
$(z-a)_\tau=T(z-a)$, we obtain
\begin{equation}\label{eq-wrk-1}
\int_{D_\tau(a,r)}h'(z_\tau)\,(\overline{z}_\tau-\overline{a}_\tau)^{m-1}\,dm_2(z_\tau)=
\frac{\pi r^{2m}}{m!}h^{(m)}(a_\tau).
\end{equation}
Since
$h'(z_\tau)=-\tau^{-1}\overline{\partial}f(z)=-\tau^{-1}\overline{\partial}f_a(z)$,
where $f_a(z)=f(z)-f(a)$, we have
\begin{align*}
\frac{\pi
r^{2m}}{m!}h^{(m)}(a_\tau)=&-\frac1\tau\int_{D_\tau(a,r)}\overline{\partial}f_a(z)\,
(\overline{z}_\tau-\overline{a}_\tau)^{m-1}\,dm_2(z_\tau)\\
=& -\frac1\tau
\int_{D_\tau(a,r)}\big(\overline\partial\big(f_a(z)(\overline{z}_\tau-\overline{a}_\tau)^{m-1}\big)-
f_a(z)\overline\partial(\overline{z}_\tau-\overline{a}_\tau)^{m-1}\big)\,dm_2(z_\tau)\\
=&-\frac{1-\tau^2}{2i\tau}\int_{C_\tau(a,r)}f_a(z)\,(\overline{z}_\tau-\overline{a}_\tau)^{m-1}\,dz\\
&+
\frac{m-1}{\tau}\int_{D_\tau(a,r)}f_a(z)\,(\overline{z}_\tau-\overline{a}_\tau)^{m-2}\,dm_2(z_\tau),
\end{align*}
where $C_\tau(a,r)=\{z\colon |(z-a)_\tau|=r\}$. Both items in the last sum
may be estimated directly, so that
$$
\bigg|\frac{\pi r^{2m}}{m!}h^{(m)}(a_\tau)\bigg|\leqslant
\frac{\pi r^m\omega(f,r)}{\tau}\bigg(1+\tau+\frac{2m-2}{m}\bigg)\leqslant \frac{3\pi
r^m}{\tau}\omega(f,r),
$$
which yields the desired estimate when we take $r=d_\tau$.
\end{proof}

\begin{remark}
In the proof of Lemma~\ref{lem-solest} one may use \cite{BagFed2017psim},
Lemma~1, that gives the desired estimates for $h'$ and $g'$. We can continue
the proof of Lemma~\ref{lem-solest} by putting this estimate into
\eqref{eq-wrk-1} and estimating the resulting integral in a suitable way.
Doing this one can show even a bit stronger estimates, than
\eqref{eq-estsol-h} and \eqref{eq-estsol-g}, namely the multiplier $m!$ in
\eqref{eq-estsol-h} and \eqref{eq-estsol-g} can be replaced with
$(m-1)!\sqrt{m}$.
\end{remark}

As a corollary of this lemma one can prove the following statement that was
obtained in a slightly different way in \cite{Zai2006psim}, Proposition~1.

\begin{corollary}\label{cor-anarc}
Let $G$ be a bounded simply connected domain in $\mathbb C$ such that its
boundary $\varGamma$ contains an analytic arc $\varUpsilon$, none of whose
points are cluster points for the set $\varGamma\setminus\varUpsilon$. Let
$\tau\in[0,1)$. Then $G$ is not $\mathcal L_\tau$-regular, and the $\mathcal
L_\tau$-Dirichlet problem in $G$ with the boundary function $1\big/(z-a)$ is
unsolvable, for any point $a\in G$ lying sufficiently close to $\varUpsilon$.
\end{corollary}

\begin{proof}
We start with the general case when $0<\tau<1$.

Let $a\in G$. Arguing by contradiction, let us assume that there exists a
function $f\in C(\overline{G})\cap\mathcal O_\tau(G)$ such that
$f|_{\varUpsilon}=1\big/(z-a)$. By \eqref{eq-soltau-nse} we have
$f(z)=g(z)+h(z_\tau)$, where $g$ and $h$ are two holomorphic functions in $G$
and $TG$, respectively.

Let $S$ be a Schwarz function of $\varUpsilon$, that is $S$ is the
holomorphic function in a neighborhood $V$ of $\varUpsilon$ such that
$\overline{z}=S(z)$ for all $z\in\varUpsilon$. It is clear, that such
function exists for any analytic curve or arc. See \cite{Dav1974book}, where
one can find an interesting introductory survey concerning the concept of a
Schwarz functions. Put $S_\tau(z):=z-\tau S(z)$, so that $z_\tau=S_\tau(z)$
for all $z\in\varUpsilon$. It can be readily verified that $S_\tau(z)\in TG$
for all $z$ lying in $G$ sufficiently close to $\varUpsilon$. Indeed, let
$\zeta\in\varUpsilon$. Since $S'_\tau(\zeta)=1-\tau S'(\zeta)$ and
$|S'(\zeta)|=1$, then $S'_\tau(\zeta)\neq0$ and therefore $S_\tau$ is
univalent in some neighborhood of $\zeta$. Let $\gamma$ be some subarc of
$\varUpsilon$ ending at the point $\zeta$ and let $\gamma'\subset
G\cup\{\zeta\}$ be some Jordan arc ending at $\zeta$ and non-tangential to
$\gamma$. So, $\angle_\zeta(\gamma,\gamma')\in(0,\pi)$, where
$\angle_\zeta(\gamma,\gamma')$ stands for the angle between $\gamma$ and
$\gamma'$ at $\zeta$. Since $S_\tau$ is univalent in a neighborhood of
$\zeta$, then
$\angle_{\zeta_\tau}(S_\tau(\gamma),S_\tau(\gamma')=\angle_\zeta(\gamma,\gamma')$.
The mapping $T\colon z\mapsto z_\tau$ is sense-preserving and hence
$\angle_{\zeta_\tau}(T\gamma,T\gamma')\in(0,\pi)$. Since
$T\gamma=S_\tau(\gamma)$, then
$\angle_{\zeta_\tau}(S_\tau(\gamma'),T\gamma)\in(0,\pi)$. Finally, since
$T\gamma'\subset TG\cup\{\zeta_\tau\}$, and since $TG$ is a Carath\'eodory
domain, then $S_\tau(\gamma')\subset TG\cup\{\zeta_\tau\}$ whenever the
length of $\gamma'$ is sufficiently small.

Let now $z\in G$ and $d=\dist(z,\varUpsilon)$. Then there exists two numbers
$C_1>0$ and $C_2>0$, independent on $d$, such that for all sufficiently small
$d$ the points $S_\tau(z)$ and $z_\tau$ can be join by some rectifiable curve
$J\subset TG$ with $\length(J)\leqslant C_1d$ and
$\dist(J,T\varUpsilon)\geqslant C_2d$. Thus
$$
h(S_\tau(z))-h(z_\tau)=\int_Jh'(\zeta_\tau)\,d\zeta_\tau
$$
and hence
$$
|h(S_\tau(z))-h(z_\tau)|\leqslant C_1d\max_{\zeta\in J}|h'(\zeta_\tau)|.
$$
According to Lemma~\ref{lem-solest} this gives $h(S_\tau(z))-h(z_\tau)\to0$
as $d\to0$. Therefore
$$
g(z)-h(S_\tau(z))-\frac1{z-a}\to0,\qquad\text{as}\quad d\to0,
$$
and hence, according to Luzin--Privalov boundary uniqueness theorem, we have
$$
g(z)-h(S_\tau(z))=\frac1{z-a}
$$
for all $z\in G$ sufficiently close to $\varUpsilon$. It remains to take a
domain $G_0\subset G$ such that $\varUpsilon\subset\partial G_0$ and the
function $h\circ S_\tau$ is holomorphic in $G_0$, and, finally, to take $a\in
G_0$. Then we arrive to a contradiction, because the function $g-h\circ
S_\tau$ is holomorphic in a neighborhood of $a$, but $1\big/(z-a)$ has a pole
therein.

The case $\tau=0$ was considered in \cite{CarParFed2002sbm}, Proposition~5.2.
The proof in this case is more simple. Indeed, the function $f$ has now the
form $f(z)=\overline{z}f_1(z)+f_0(z)$, where $f_0,f_1$ are holomorphic
functions in $G$, and
$f(z)-(f_0(z)+S(z)f_1(z))=f_1(z)(\overline{z}-S(z))\to0$ uniformly as
$z\to\varUpsilon_0$, $z\in G\cap V$, for some subarc
$\varUpsilon_0\subset\varUpsilon$. Then $f_1(z)S(z)+f_0(z)$ coincides with
$1\big/(z-a)$ in $G\cap V$, which is clearly impossible.
\end{proof}

\subsection*{Domains in $\mathbb C$ and their conformal mappings}

We recall that a Jordan curve $\varGamma$ is a homeomorphic image of the unit
circle $\mathbb T$, and an arc is a homeomorphic image of a straight line
segment. By virtue of the classical Jordan curve theorem, the set $\mathbb
C\setminus\varGamma$ is not connected. It consists of two connected
components $D(\varGamma)$ and $D_{\infty}(\varGamma)$, where $D(\varGamma)$
is the bounded one. The domain $D(\varGamma)$ is called a Jordan domain
bounded by $\varGamma$. Moreover, one has $\varGamma=\partial
D(\varGamma)=\partial D_\infty(\varGamma)$. It is clear, that every Jordan
domain is simply connected.

Following \cite{Pom1992book} we say, that a curve $\varGamma$ (which may be
both an arc, or a Jordan curve) is of class $C^n$, $n=1,2,\ldots$, if it has
a parametrization $\varGamma\colon w(\xi)$, $0\leqslant\xi\leqslant1$, which
is $n$ times continuously differentiable and satisfies $w'(\xi)\neq0$ for
$\xi\in[0,1]$. The curve $\varGamma$ is of class $C^{n,\alpha}$ where
$0<\alpha\leqslant1$, if moreover, this parametrization possesses the
property
$$
|w^{(n)}(\xi_1)-w^{(n)}(\xi_2)|\leqslant C|\xi_1-\xi_2|^{\alpha},
\quad\text{for}\quad \xi_1,\xi_2\in[0,1].
$$
If $\varGamma$ is a Jordan curve of class $C^{n,\alpha}$, and if
$G=D(\varGamma)$ is a Jordan domain bounded by $\varGamma$, one says that $G$
is a Jordan domain with the boundary of class $C^{n,\alpha}$.

Let now $G=D(\varGamma)$ be some Jordan domain in the complex plane bounded
by a Jordan curve $\varGamma$, and let $\varphi$ be some conformal map from
$\mathbb D$ onto $G$. According to the classical Carath\'eodory extension
theorem (see, for instance, \cite{Pom1992book}, Theorem~2.6), the function
$\varphi$ can be extended to the homeomorphism from $\overline{\mathbb D}$
onto $\overline{G}$. We will keep the notation $\varphi$ for this extended
homeomorphism. The following Kellogg--Warschawski theorem, see
\cite{Pom1992book}, Theorem~3.6, says that for any Jordan domain $G$ with the
boundary of class $C^{n,\alpha}$ the function $\varphi$ has the following
smoothness property:
\begin{atheorem}\label{thm-kelwar}
Let $\varphi$ map $\mathbb D$ conformally onto the inner domain of the Jordan
curve $\varGamma$ of class $C^{n,\alpha}$ where $n=1,2,\ldots$ and
$0<\alpha<1$. Then $\varphi^{(n)}$ has a continuous extension to
$\overline{\mathbb D}$ and
\begin{equation}\label{eq-confmaplip}
|\varphi^{(n)}(z_1)-\varphi^{(n)}(z_2)|\leqslant C|z_1-z_2|^\alpha,
\quad\text{for}\quad z_1,z_2\in\overline{\mathbb D}.
\end{equation}
\end{atheorem}

The following proposition is the direct consequence of
Theorem~\ref{thm-kelwar} and the Cauchy integral formula.

\begin{corollary}\label{cor-x1}
Let $\alpha\in(0,1)$, let $G$ be a Jordan domain with the boundary
$\varGamma$ of class $C^{1,\alpha}$, and let $\varphi$ maps $\mathbb D$
conformally onto $G$. Then for every $z\in\mathbb D$ one has
\begin{equation}\label{eq-confmapest}
|\varphi''(z)|\leqslant\frac{C}{(1-|z|)^{1-\alpha}},\quad z\in\mathbb D.
\end{equation}
\end{corollary}

\subsection*{Main result and scheme of its proof}

As noted above, the problem of $\mathcal L$-regularity of a given domain $G$
is equivalent to the problem of $\mathcal L_\tau$-regularity of the domain
$TG$ for $\tau=\tau(\mathcal L)$. It is worth to note that the lack of
invariance of $\mathcal L$ (and even $\mathcal L_\tau$) under transformations
that change angles, takes no effect to the forthcoming constructions and
arguments, since we are dealing with the class of domains with
$C^{1,\alpha}$-smooth boundaries.

\begin{theorem}\label{thm-dirprob}
Let $\alpha\in(0,1)$, and let $G$ be a Jordan domain with the boundary
$\varGamma$ of class $C^{1,\alpha}$. Then, for every $\tau$,
$0\leqslant\tau<1$, the domain $G$ is not $\mathcal L_\tau$-regular.
\end{theorem}

\begin{proof}
For $\tau=0$ this theorem was proved in \cite{Maz2009sbm}. Thus, in the rest
of the proof we assume that $\tau>0$. Let $\varphi$ be some conformal mapping
from $\mathbb D$ onto $G$ which is assumed already extended to the
corresponding homeomorphism from $\overline{\mathbb D}$ to $\overline{G}$.
Define the class of functions
$$
\mathcal K_\tau=\big\{F\colon F(z)=f(\varphi(z))\ \text{for}\
f\in C(\overline{G})\cap\mathcal O_\tau(G)\big\}.
$$
In view of \eqref{eq-soltau-nse} every function $F\in\mathcal K_\tau$ has the
form $F(z)=h(\varphi(z)-\tau\overline{\varphi(z)})+g(\varphi(z))$, where $g$
and $h$ are holomorphic functions in $G$ and $TG$, respectively. We will
prove not only the fact that $C(\mathbb T)\neq \mathcal K_{\tau,\mathbb
T}=\{F|_{\mathbb T}\colon F\in\mathcal K_\tau\}$, but we will establish that
$\mathcal K_{\tau,\mathbb T}$ is a Baire first category set in $C(\mathbb
T)$. As usual, $F|_{\mathbb T}$ stands for the restriction of $F$ to $\mathbb
T$.

We need the following result that will be established in
Section~\ref{s:proof} below: There exists a family $\{M_n\}_{n=1}^{\infty}$
of functionals defined on the space $C(\mathbb T)$ satisfying the following
properties
\begin{trivlist}
\item[\quad]
1)~there exists an absolute constant $\mu_0>0$ such that for every positive
integer $n$
\begin{equation}\label{eq-Mn-norm}
\|M_n\|=\mu_0=M_n(-i\overline{z}^{n+1});
\end{equation}
\item[\quad]
2)~for every function $F\in K_\tau$, $F=f\circ\varphi$, $f\in
C(\overline{G})\cap\mathcal O_\tau(G)$, and for every positive integer $n$
large enough we have
\begin{equation}\label{eq-Mn-estonsol}
|M_n(F)|\leqslant\gamma_n\|f\|_{\overline{G}},
\end{equation}
where $\gamma_n=\gamma_n(G,\tau)$ and $\gamma_n\to0$ as $n\to\infty$;
\item[\quad]
3)~for every trigonometric polynomial $P$ of degree $\nu$ and for every
integer $n>\nu$ large enough we have
\begin{equation}\label{eq-Mn-estontp}
|M_n(P)|\leqslant\gamma_{n-\nu}\|P\|_{\overline{\mathbb T}}.
\end{equation}
We recall, that $P$ is a function of the form
$P(z)=\sum_{k=-\nu}^{\nu}c_kz^k$, and $P|_{\mathbb
T}=\sum_{k=-\nu}^{\nu}c_ke^{ik\vartheta}$, where $\nu$ is a positive integer
and $c_k$, $-\nu\leqslant k\leqslant\nu$, are complex numbers (coefficients).
\end{trivlist}
It is crucial that $\gamma_n$ in \eqref{eq-Mn-estonsol} and
\eqref{eq-Mn-estontp} depends only on $G$ and $\tau$.

Take a number $H>0$ and consider the set
$$
\mathcal K_{\tau,\mathbb T,H}=\{F|_{\mathbb T}\colon F\in\mathcal K_\tau,\
\|F\|_{\overline{\mathbb D}}\leqslant H\}.
$$
For an arbitrary $\psi\in C(\mathbb T)$ and $\delta>0$, let $B(\psi,\delta)$
be the ball in the space $C(\mathbb T)$ with center $\psi$ and radius
$\delta$. We are going to prove that $\mathcal K_{\tau,\mathbb T,H}$ is not
dense in $B(\psi,\delta)$. Assume that $\|\psi\|_{\mathbb T}=1$ and take a
trigonometric polynomial $P$ of degree $\nu$ such that $\|\psi-P\|_{\mathbb
T}\leqslant\delta/3$. It follows from \eqref{eq-Mn-estontp} that
$|M_n(P)|\leqslant C_1\gamma_{n-\nu}$ if $n>\nu$ is large enough. Let
$P_0(z)=P(z)-i\delta z^{-(n+1)}/3$. Since $\gamma_n\to0$ as $n\to\infty$, one
can find $n_1$ such that $\gamma_{n-\nu}<\delta\mu_0/(6C_1)$ for $n>n_1$.
Therefore, for $n>n_1$ we have $|M_n(P_0)|\geqslant\delta\mu_0/6$. It yields
that for every function $\chi\in B(P_0,\delta/12)\subset B(\psi,\delta)$ we
have $|M_n(\chi)|\geqslant\delta\mu_0/12$ for $n>n_1$. On the other hand,
since $\gamma_n\to0$ as $n\to\infty$, there exists a positive integer $n_2$
such that $\gamma_n<\delta\mu_0/(20H)$ for every integer $n>n_2$. Thus,
$|M_n(F)|<\delta\mu_0/20$ for every integer $n>n_2$ and for every function
$F\in\mathcal K_{\tau,\mathbb T,H}$. This yields that the ball
$B(P_0,\delta/12)$ is contained in $B(\psi,\delta)$ and does not contain any
function from the space $\mathcal K_{\tau,\mathbb T,H}$.
\end{proof}


\section{Theorem~\ref{thm-mainest} and its proof}
\label{s:proof}

Within this section $G$ will denote a Jordan domain in $\mathbb C$ with the
boundary $\varGamma$, and $\varphi$ will denote some conformal mapping from
$\mathbb D$ onto $G$ which is already considered extended to the
corresponding homeomorphism from $\overline{\mathbb D}$ onto $\overline{G}$.

Our main aim in this section is to prove that there exists a family of
functionals $M_n$ on the space $C(\overline{G})\cap\mathcal O_\tau(G)$ for
which the properties 1)--3) used in the proof of Theorem~\ref{thm-dirprob}
are satisfied for some $\gamma_n=\gamma_n(G,\tau)\to0$ as $n\to\infty$.

Take a sufficiently small number $\varepsilon>0$ whose value will be
specified later, and for a given point $\zeta\in\mathbb T$ let us take the
positive real-valued function $\varPsi_{\varepsilon,\zeta}\in
C^\infty_0(D(\zeta,\varepsilon))$ such that
$$
\int\varPsi_{\varepsilon,\zeta}(z)\,dm_2(z)=1
$$
and
$$
\mu_0=\int_{\mathbb T}\varPsi_{\varepsilon,\zeta}(z)\,|dz|>0.
$$
Using this function, for every integer $n\geqslant0$ we define the functional
$$
\mathcal M_{n,\varepsilon,\zeta}\colon F\mapsto
\int_{\mathbb T}\varPsi_{\varepsilon,\zeta}(z)\,F(z)\,z^n\,dz
$$
acting on the space $C(\mathbb T)$. Since $\overline{z}z=|z|^2=1$ on $\mathbb
T$ we have
\begin{align*}
|\mathcal M_{n,\varepsilon,\zeta}(F)|\leqslant&
\|F\|_{\mathbb T}\int\varPsi_{\varepsilon,\zeta}(z)\,|dz|=\mu_0\|F\|_{\mathbb T},\\
\intertext{for $F\in C(\mathbb T)$, and}
\mathcal M_{n,\varepsilon,\zeta}(-i\overline{z}^{n+1})= &\int_{\mathbb
T}\varPsi_{\varepsilon,\zeta}(z)\,|dz|=\mu_0,
\end{align*}
which gives \eqref{eq-Mn-norm}.

The estimate \eqref{eq-Mn-estonsol} is the consequence of the following
result, the proof of which is the main aim of this section.

\begin{theorem}\label{thm-mainest}
Let $\alpha\in(0,1)$ and let $G$, $\varGamma$ and $\varphi$ be as mentioned
above. Assume that $\varGamma$ is of class $C^{1,\alpha}$. Moreover, suppose
that $0\in\varGamma$, $\varphi(1)=0$, and the tangent line to $\varGamma$ at
the origin coincides with the real axis.

Then for every $\tau\in(0,1)$ there exist such point $\zeta\in\mathbb T$ and
numbers $A=A(\tau,G)>0$ and $\varepsilon>0$, that for each function $f\in
C(\overline{G})\cap\mathcal O_{\tau}(G)$ and for every sufficiently large
$n\in\mathbb N$ the following inequality is satisfied
\begin{equation}\label{eq-mainest}
\mathcal M^*_{n,\varepsilon,\zeta}(f)=
\mathcal M_{n,\varepsilon,\zeta}(f\circ\varphi|_{\mathbb T})\leqslant
A\frac{\|f\|_{\overline{G}}}{n^{\alpha/2}}.
\end{equation}
\end{theorem}

To prove this theorem we need several technical lemmas. Let us recall that
the real-linear transformation $T\colon\mathbb C\mapsto\mathbb C$ is defined
in such a way that $Tz=z_\tau=z-\tau\overline{z}$.

\begin{lemma}\label{lem-confmap}
Let $G$ be a Jordan domain in $\mathbb C$ with the boundary $\varGamma$, and
let $\varphi$ be some conformal mapping from $\mathbb D$ onto $G$. Assume
that $\varphi\in C^1(\overline{\mathbb D})$ and $\varphi'(z)\neq0$ as
$z\in\varGamma$. Suppose moreover, that $0=\varphi(1)\in\varGamma$, and the
tangent line to $\varGamma$ at the origin is the real line. Then there exists
$\varepsilon>0$ such that for every point $a\in
B_\varepsilon=\varphi(D(1,\varepsilon)\cap\mathbb D)$ the following estimate
takes place
$$
|T(a-b)|\geqslant (1+\tau)(1-\varepsilon)(1-|a_0|)
\min_{z\in D(1,\varepsilon)\cap\mathbb D}|\varphi'(z)|,
$$
where $a_0=\varphi^{-1}(a)$ and $b\in\varGamma$ is the nearest point to $a$.
\end{lemma}

\begin{proof}
For notation simplification we put $D_\varepsilon=D(1,\varepsilon)$ and
$W_\varepsilon=D_\varepsilon\cap\mathbb D$. Also we put $a=\xi_0+i\eta_0$.
Consider a sufficiently small arc $\gamma$ of $\varGamma$ containing the
origin such that $\gamma$ can be parameterized by the equation
$\eta=\psi(\xi)$, $\zeta=\xi+i\eta\in\gamma$ (we have used here the fact that
$\varGamma$ is a smooth curve and $\varphi'(z)\neq0$ for $z\in\varGamma$).
Taking $\varepsilon$ small enough we obtain that
$\gamma\subset\varphi(D_\varepsilon\cap\mathbb T)$, and the point
$c:=\xi_0+i\psi(\xi_0)\in\varGamma$, see Fig.~\ref{fig-lemconfmap}.

\begin{figure}[H]
\includegraphics[width=100mm]{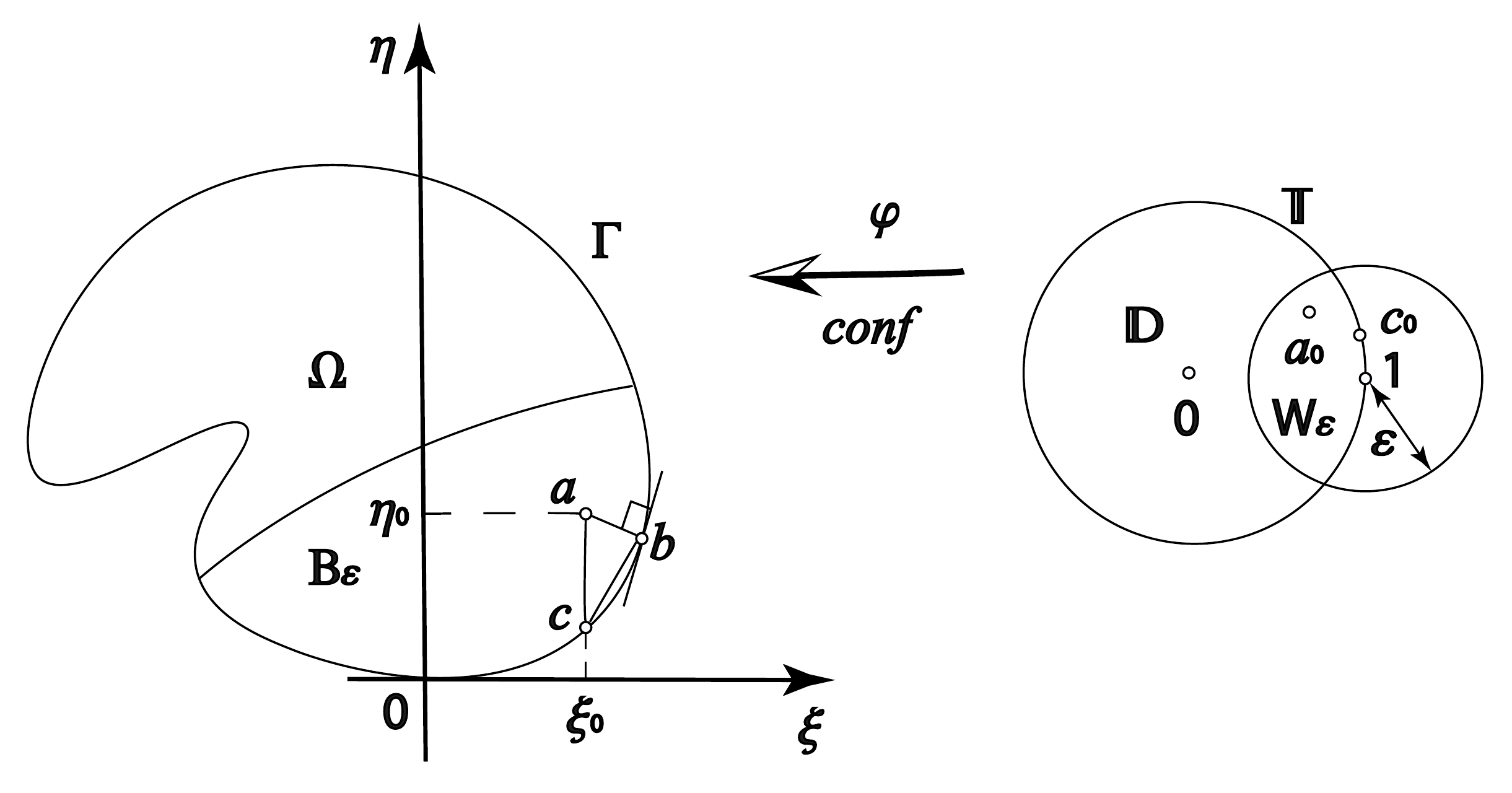}
\caption{Construction from Lemma~\ref{lem-confmap}}
\label{fig-lemconfmap}
\end{figure}

For an arbitrary point $\zeta\in\gamma$ the quantity
$|a-\zeta|^2=(\xi_0-\xi)^2+(\eta_0-\psi(\xi))^2$ attains its minimum at the
point $\xi+i\psi(\xi)$ where $\xi$ is such that
$(\xi_0-\xi)+\psi'(\xi)(\eta_0-\psi(\xi))=0$. Notice that the latter equation
is the equation of normal to $\gamma$ passed from the point $a$. Thus, the
point $b\in\gamma$ nearest to $a$ belongs to the normal to $\gamma$ passing
from $a$. Notice, that without loss of generality we may assume that
$\xi>\xi_0$.

Denote by $\varDelta$ the triangle with vertexes at the points $a$, $b$ and
$c$, and denote the angles at the vertices $a$, $b$ and $c$ of this triangle
by $\beta(a)$, $\beta(b)$ and $\beta(c)$, respectively. Since the points $a$
and $b$ belong to the normal to $\varGamma$, then $b-a=k(\psi'(\xi)-i)$ with
$k\in\mathbb R$. Thus
$$
\sin\beta(a)=\frac{|\psi'(\xi)|}{\sqrt{1+\psi'(\xi)^2}},\quad\text{and}
\quad
\cos\beta(c)=\frac{\psi'(\widetilde\xi)}{\sqrt{1+\psi'(\widetilde\xi)^2}},
$$
where $\xi_0\leqslant\widetilde\xi\leqslant\xi$ (according to Lagrange's mean
value theorem). Since $\psi'(0)=0$ and since $\xi-\xi_0$ is small when
$\varepsilon$ is small enough, then both quantities $\sin\beta(a)$ and
$\cos\beta(c)$ are also small and hence $\beta(a)\to0$ and
$\beta(c)\to\dfrac\pi2$ as $\varepsilon\to0$. Therefore
$\beta(b)\to\dfrac\pi2$ as $\varepsilon\to0$. Applying the sine theorem to
the triangle $\varDelta$ we obtain
$$
\frac{|c-b|}{\sin\beta(a)}=\frac{|a-c|}{\sin\beta(b)},
$$
which gives that
$$
\frac{|c-b|}{|a-c|}=\frac{\sin\beta(a)}{\sin\beta(b)}\to0
$$
as $\varepsilon\to0$. Hence, for sufficiently small $\varepsilon$ we have
that $|c-b|<\varepsilon|a-c|$ for all $a\in B_\varepsilon$. Moreover, for all
such $a$ we have
$$
|(a-b)_\tau|\geqslant(1+\tau)\big||a-c|-|b-c|\big|\geqslant
(1+\tau)(1-\varepsilon)|a-c|.
$$
It remains to use this estimate together with the following one
$$
|a-c|=|\varphi(a_0)-\varphi(c_0)|
\geqslant|a_0-c_0|\min_{z\in W_\varepsilon}|\varphi'(z)|
\geqslant(1-|a_0|)\min_{z\in W_\varepsilon}|\varphi'(z)|,
$$
where $c_0=\varphi^{-1}(c)$. The lemma is proved.
\end{proof}

\begin{remark}\label{rem-domain}
In the context of the problem under consideration we may assume that for any
Jordan domain $B$ with smooth boundary the origin belongs to $\partial B$ and
the tangent line to $B$ at this point is the real line. Indeed for any such
domain $B$ the set $\partial B$ contains a point $w$ having minimum ordinate
along $\partial B$. The tangent line to $\partial B$ at $w$ is horizontal. It
remains to use shift moving $w$ to the origin, and recall that the operator
$\mathcal L_\tau$ is invariant under such transformation of the plane (see
Remark~\ref{rem-change}).
\end{remark}

Combining Lemmas~\ref{lem-solest} and~\ref{lem-confmap} together we have the
next proposition that contains our central estimates.

\begin{lemma}\label{lem-estim}
Suppose all conditions of Lemmas~\ref{lem-solest} and~\ref{lem-confmap} to be
satisfied, and let $\varepsilon$ is taken from Lemma~\ref{lem-confmap}. Then
for all $z\in W_\varepsilon=D(1,\varepsilon)\cap\mathbb D$ we have
\begin{align}
\big|\varphi'(z)^mh^{(m)}(T\varphi(z))\big|&\leqslant
C\frac{m!R_\varepsilon^m}{(1+\tau)^m(1-\varepsilon)^m(1-|z|)^m}\|f\|_{\overline{G}},\label{eq-est-h}\\
\big|\varphi'(z)^mg^{(m)}(\varphi(z))\big|&\leqslant
C\frac{m!R_\varepsilon^m}{(1-|z|)^m}\|f\|_{\overline{G}},
\end{align}
where $R_\varepsilon=\dfrac{\max\{|\varphi'(z)\colon z\in
W_\varepsilon\}}{\min\{|\varphi'(z)\colon z\in W_\varepsilon\}}$.
\end{lemma}

In fact we need to strengthen the estimates obtained in Lemma~\ref{lem-estim}
in the case where the initial function $f\in C(\overline{G})\cap\mathcal
O_\tau(G)$ is $\mathcal L_\tau$-analytic in some neighborhood of
$\overline{G}$. Namely, the next proposition takes place.

\begin{lemma}\label{lem-estim-reg}
Let $U$ be an open set such that $\overline{G}\subset U$ and let
$f\in\mathcal O_\tau(U)$. Assume that all conditions of
Lemmas~\ref{lem-solest} and~\ref{lem-confmap} are fulfilled. Then for all
$z\in W_\varepsilon=D(1,\varepsilon)\cap\mathbb D$ the estimate is satisfied
\begin{equation}\label{eq-est-h-reg}
\big|\overline{\varphi'(z)}{}^mh^{(m)}(T\varphi(z))\big|\leqslant
Cm!\bigg(\frac{1+\beta(\varepsilon)}{1+\tau}\bigg)^m\dfrac{\rho_m(z)}{(1-|z|)^m}\|f\|_U,
\end{equation}
where $\beta(\varepsilon)\to0$ as $\varepsilon\to0$, and
$$
\rho_m(z)=\bigg(\frac{1-|z|}{R-|z|}\bigg)^m
$$
for some number $R>1$.
\end{lemma}

To verify this lemma we need to apply Lemma~\ref{lem-estim} considering a
conformal mapping from the disk $D(0,R)$ for some $R>1$ onto $G$ instead of
$\varphi$ and taking into account the fact that
$R_\varepsilon\big/(1-\varepsilon)$ in this case is
$C_1(1+\beta(\varepsilon))$ with $\beta(\varepsilon)\to0$ as $\varepsilon\to
0$ and with certain constant $C_1>0$.

The next simple statement may be readily verified using Green's formula and
integration by parts taking into account the facts that $1-|z|^2=0$ for
$z\in\mathbb T$ and $\overline\partial(1-|z|^2)=-z$, see \cite{Maz2009sbm},
formula~(2.3).
\begin{lemma}\label{lem-green}
Let $F\in C^\infty(D(0,R))$ for some $R>1$. Then for any $k,N\in\mathbb N$,
one has
\begin{equation}\label{eq-m1}
\int_{\mathbb D}(1-|z|^2)^{k-1}\,F(z)\,z^{N-k+1}\,dm_2(z)=
\frac1k\int_{\mathbb D}(1-|z|^2)^k\,z^{N-k}\,\overline\partial F(z)\,dm_2(z).
\end{equation}
\end{lemma}

\begin{proof}[Proof of Theorem~\ref{thm-mainest}]

According to Remark~\ref{rem-domain} one may (and shall) assume that $G$ and
$\varphi$ satisfy all conditions of Lemma~\ref{lem-confmap}. Thus take
$\varepsilon$ from this lemma and assume that $\zeta=1$.

Take a function $f\in C(\overline{G})\cap\mathcal O_\tau(G)$. According to
\eqref{eq-soltau-nse} one has $f(z)=h(z_\tau)+g(z)$, $z\in G$, where $g$ and
$h$ are holomorphic functions in $G$ and $TG$, respectively. Let us assume
for a moment that $f\in\mathcal O_\tau(U)$ for some open set $U$ that
contains $\overline{G}$, so that the functions $g$ and $h$ are holomorphic in
$U$ and $TU$, respectively. We will argue in the frameworks of this
assumption. At the last step of the proof it remains to apply the
regularization arguments based in the fact that the initial function $f$ can
be approximated uniformly on $\overline{G}$ by functions $\mathcal
L_\tau$-analytic in neighborhoods of $\overline{G}$ (each function in its own
neighborhood).

In what follows we will use the following notations. For $k\in\mathbb N$ we
put $\mu_k(z)=\overline{\varphi'(z)}^k$, $z\in\overline{\mathbb D}$. We will
write $A\lesssim B$ if $A\leqslant CB$ for some number $C>0$ which may depend
on $\|\varPsi_\varepsilon\|$, $\|\overline\partial\varPsi_\varepsilon\|$,
$\min_{z\in\overline{\mathbb D}}|\varphi'(z)|$. Similarly, all constants in
the usual ``O-big'' notation $\mathop{\mathbb O}(\cdot)$ may (and will)
depend on these quantities. Moreover, for $k\in\mathbb N$ we put
$$
F_k(z):=\frac{(-\tau)^k}{(k-1)!}\varPsi_\varepsilon(z)\,\mu_k(z)\,h^{(k)}(T\varphi(z)).
$$

Direct computations based on the standard Green's formula give that
\begin{multline}\label{eq-est-zerostep}
\mathcal M^*_{n,\varepsilon,\zeta}(f)= \int_{\mathbb
T}\varPsi_\varepsilon(z)\,f(\varphi(z))\,z^n\,dz= 2i\int_{\mathbb
D}z^n\overline\partial[\varPsi_\varepsilon(z)f(\varphi(z))]\,dm_2(z)=\\
=2i\int_{\mathbb
D}z^nf(\varphi(z))\,\overline\partial\varPsi_\varepsilon(z)\,dm_2(z)-
2i\int_{\mathbb D}z^nF_1(z)\,dm_2(z).
\end{multline}

It is clear that
$$
\bigg|\int_{\mathbb
D}z^nf(\varphi(z))\,\overline\partial\varPsi_\varepsilon(z)\,dm_2(z)\bigg|\leqslant
\dfrac{2\pi\|\varPsi'_\varepsilon\|\,\|f\|}{n+2},
$$
and we need to estimate only the second summand in the right-hand side of
\eqref{eq-est-zerostep}. This estimate requires much more delicate
considerations. The main idea how to estimate the quantity
$\displaystyle\int_{\mathbb D}z^nF_1(z)\,dm_2(z)$, is to apply
Lemma~\ref{lem-green} to the functions $F_k$, $k\in\mathbb N$, consequently.
Take an arbitrary $m\in\mathbb N$. Since
$$
\overline\partial h^{(k)}(T\varphi(z))=-\tau\overline{\varphi'(z)}h^{(k+1)}(T\varphi(z))
$$
and $\mu_k(z)\overline{\varphi'(z)}=\mu_{k+1}(z)$, we have
\begin{align*}
\int_{\mathbb D}F_1(z)z^n\,dm_2(z)=&\int_{\mathbb D}(1-|z|^2)z^{n-1}\,\overline\partial F_1(z)\,dm_2(z)\\
=&-\tau\int_{\mathbb D}(1-|z|^2)z^{n-1}\,\overline\partial[\varPsi_\varepsilon(z)\mu_1(z)]
h'(T\varphi(z))\,dm_2(z)\\
&+\int_{\mathbb D}(1-|z|^2)z^{n-1}\,F_2(z)\,dm_2(z)\\
=&\cdots\;\cdots\;\cdots\\
=&\sum_{k=1}^{m}\frac{(-\tau)^k}{k!}\int_{\mathbb D}(1-|z|^2)^kz^{n-k}
\overline\partial[\varPsi_\varepsilon(z)\mu_k(z)]h^{(k)}(T\varphi(z))\,dm_2(z)\\
&+\int_{\mathbb D}(1-|z|^2)^mz^{n-m}\,F_{m+1}(z)\,dm_2(z).
\end{align*}
Let us estimate the integrals
$$
I_k=\frac{(-\tau)^k}{k!}\int_{\mathbb D}(1-|z|^2)^kz^{n-k}
\overline\partial[\varPsi_\varepsilon(z)\mu_k(z)]h^{(k)}(T\varphi(z))\,dm_2(z),
\quad  k\in\mathbb N.
$$
Observe that for $p\in\mathbb N$ and $\nu\in(0,1)$ it holds
\begin{equation}\label{eq-tech1}
\int_0^1\dfrac{r^p\,dr}{(1-r)^{1-\nu}}\approx\dfrac1{\nu p^\nu}.
\end{equation}

Let now $\beta(\varepsilon)$ be taken from Lemma~\ref{lem-estim-reg}, so that
$\beta(\varepsilon)\to0$ as $\varepsilon\to0$. Then for sufficiently small
$\varepsilon$ we have
\begin{equation}\label{eq-est-beta}
\frac{2\tau(1+\beta(\varepsilon))}{1+\tau}<\tau_1<1
\end{equation}
for some $\tau_1\in(0,1)$ which may depend only on $\tau$ and $G$ (of course
implicitly, via $\varphi$, $\varepsilon$, etc.). Since
$(1-|z|^2)^k\leqslant2^k(1-|z|)^k$ for $z\in\mathbb D$, then using
\eqref{eq-est-h-reg} and \eqref{eq-est-beta} we have
\begin{equation}\label{eq-tech2}
|I_k|\leqslant C\tau^k\bigg(\frac{1+\beta(\varepsilon)}{1+\tau}\bigg)^k 2^k
\int_{D(z_0,\varepsilon)}(1-|z|)^k\frac{k}{\varphi'(z)}\frac{\rho_k(z)}{(1-|z|)^{k+1-\alpha}}|z|^{n-k}\,dm_2(z).
\end{equation}

For $k<n$ this inequality together with \eqref{eq-tech1} gives that
$$
|I_k|=\mathop{\mathbb O}\Big(\frac{k\tau_1^k}{(n-k)^\alpha}\Big).
$$
For $k\geqslant n$ in order to estimate $I_k$ we will use next arguments. For
$z\in W_\varepsilon=D(1,\varepsilon)\cap\mathbb D$ we have
$r^{n-m}\leqslant(1-\varepsilon)^{n-k}$. Moreover, for sufficiently small
$\varepsilon$ it holds that
$$
\tau_2=\frac{\tau_1}{1-\varepsilon}<1.
$$
So that for $k\geqslant n$ we have $|I_k|=\mathop{\mathbb O}(k\tau_2^k)$.

It can be readily checked that
$$
\sum_{k=1}^{n-1}\frac{k\tau_1^k}{(n-k)^\alpha}\lesssim \rho(n,\alpha,\tau),
$$
where
$$
\rho(n,\alpha,\tau_1)=\frac1{n^{\alpha/2}}+
\frac{n^{\alpha/4}\tau_1^{n^{\alpha/4}}}{1-\tau_1}+
\frac{\tau_1^{n^{\alpha/4}+1}}{(1-\tau_1)^2}.
$$
Indeed it is enough to split the sum being estimate into two sums (where the
first sum is taken over $k$ running from $0$ to the integer part of the
number $n^{\alpha/4}$, while the second one is taken over remaining  values
of $k$) and to estimate directly both sums obtained. Moreover,
$$
\sum_{k=n}^{\infty}k\tau_2^k=
\frac{\tau_2^n}{1-\tau_2}\Big(n+\frac{\tau_2}{1-\tau_2}\Big).
$$
Since the last quantity tends to zero as $n\to\infty$, and since
$\rho(n,\alpha,\tau_1)\to0$ as $n\to\infty$, then the integrals
$\displaystyle\int_{\mathbb D}(1-|z|^2)^mz^{n-m}\,F_{m+1}(z)\,dm_2(z)$ can be
made arbitrary small by taking $m$ large enough. It gives, finally, that
\eqref{eq-mainest} takes place and the proof of Theorem~\ref{thm-mainest} is
completed.
\end{proof}

The remaining estimate \eqref{eq-Mn-estontp} for $M_n=\mathcal
M_{n,\varepsilon,\zeta}$ is the consequence of the following observation: for
$P(z)=\sum_{k=-\nu}^{\nu}c_kz^k$ with integer $\nu>0$ and $c_k\in\mathbb C$,
$-\nu\leqslant k\leqslant \nu$, we have
$$
|\mathcal M_{n,\varepsilon,\zeta}(P)|=
|\mathcal M_{n-\nu,\varepsilon,\zeta}(z^\nu P)|\leqslant A(n-\nu)^{-\alpha/2}\|P\|_{\mathbb T}
$$
in view of \eqref{eq-mainest}, because for $Q=z^\nu P$ we have
$Q\circ\varphi^{-1}\in C(\overline{G})\cap\mathcal O_\tau(G)$ since
$Q\circ\varphi^{-1}$ is holomorphic in $G$.

Using the family of functionals $M_n=\mathcal M_{n,\varepsilon,\zeta}$
constructed in this section and follow the line of reasoning presented at the
end of Section~\ref{s:aux} we arrive to the complete proof of
Theorem~\ref{thm-dirprob}.

At the end of this section let us note that in the bianalytic case (that is
for $\tau=0$) the estimate \eqref{eq-mainest} can be improved a bit. Namely,
for every function $f\in C(\overline{G})\cap \mathcal
O(G,\overline\partial^2)$ and for all sufficiently large integer $n$, it
holds
$$
\bigg|\int_{\mathbb T}f(\varphi(z))\,z^n\,dz\bigg|\leqslant
A\frac{\|f\|_{\overline{G}}}{n^\alpha}.
$$
The proof of this estimate may be obtained following the same scheme that was
used in the proof of the estimate \eqref{eq-mainest}, but in this (in view of
special algebraic structure of bianalytic functions) it is enough to use
Green's formula only twice and estimate thereafter the obtained integrals
directly using \eqref{eq-confmapest} and applying Lemma~3 from
\cite{Car1985jat} instead of Lemma~\ref{lem-estim}.


\section{Outline of the proof of the second statement in Theorem~\ref{thm-mazalov}}
\label{s:example}

In this section we are going to present a schematic outline of the proof of
the following proposition which is the second statement of
Theorem~\ref{thm-mazalov}.

\begin{proposition}\label{pro-s:emx-1}
There exists a Jordan domain $G$ with Lipschitz boundary such that $G$ is
$\overline\partial^2$-regular.
\end{proposition}

This result was obtained in \cite{Maz2009sbm}, and its proof is very involved
both substantively and technically. The construction of the desired domain
$G$ is based on lacunary series technique, on variational principles of
conformal mappings, and on Rudin--Carleson theorem about interpolation peak
sets for continuous holomorphic functions.

The main aim of this section is to highlight the main steps of the
construction of $\overline\partial^2$-regular domain $G$, and to show the way
how the main difficulties of the corresponding construction may be overcome.

Let $G$ be a Jordan domain and $\varphi$ be some conformal mapping from
$\mathbb D$ onto $G$, and assume that $\varphi$ is already extended to the
eponymous homeomorphism from $\overline{\mathbb D}$ onto $\overline{G}$
according to the Carath\'eodory extension theorem. In order to satisfy the
property that $G$ is $\overline\partial^2$-regular we need to have
\begin{equation}\label{eq-s:exm-1}
\int_{\mathbb D}|\varphi''(z)|\,dm_2(z)=\infty.
\end{equation}
Indeed, otherwise the space of functions belonging to $C(\partial G)$ which
are restrictions to $\partial G$ of some functions belonging to the space
$C_{\overline\partial^2}(\overline{G})$ is a Baire first category set. This
fact is proved in \cite{Maz2009sbm}, Theorem~1; its proof may be obtained
following the same scheme that was used above to prove
Theorem~\ref{thm-dirprob} (see also the latter paragraph of
Section~\ref{s:proof}), but the proof in the bianalytic case turns out to be
rather simpler in view of some special properties bianalytic functions
possessed in contrast to $\mathcal L_\tau$ analytic ones, $\tau\in(0,1)$.

Before constructing the univalent function $\varphi$ satisfying the above
mentioned conditions, let us make one auxiliary construction. We need to
construct the function $\psi\in C(\overline{\mathbb D})$ which is holomorphic
in $\mathbb D$, such that $\psi'\in H^2(\mathbb D)$ (this condition is weaker
than the univalence one) and for which the condition \eqref{eq-s:exm-1} is
satisfied. Here $H^p(\mathbb D)$, $p>0$, is the standard Hardy spaces in the
unit disk. Let now $\psi$ be an arbitrary function that has the form
\begin{equation}\label{eq-s:exm-2}
\psi(z)=z+\sum_{k=k_0}\frac{z^{m_k}}{km_k},
\end{equation}
where $m_k$, $k\geqslant k_0$, are positive integers such that the lacunary
conditions are fulfilled
$$
\frac{m_{k+1}}{m_k}>2,\qquad \sum_{n=2}^{\infty}\sum_{k=1}^{n-1}\frac{m_k}{m_n}<\infty.
$$
It is clear that $\psi\in C(\overline{\mathbb D})$ and the condition
\eqref{eq-s:exm-1} is an immediate consequence of the fact that the series
$\sum_{k=1}^{\infty}k^{-1}$ diverges. Unfortunately, for functions of the
form \eqref{eq-s:exm-2} the condition \eqref{eq-s:exm-1} forbids completely
the univalence of $\psi$. But the following important results takes place
(see \cite{Maz2009sbm}, Lemma~3.2 and Theorem~3).

\begin{proposition}\label{pro-s:emx-2}
For every function $f\in C(\mathbb T)$ and for every function $\psi$ of the
form \eqref{eq-s:exm-2} there exists a couple $(\varPhi_1,\varPhi_2)$ of
functions holomorphic in $\mathbb D$ such that the function
$F_f=\varPhi_1\overline\psi-\varPhi_2$ is extended continuously to $\mathbb
T$ and satisfies the conditions $\sup_{z\in\mathbb D}|F_f(z)|\leqslant
2\|f\|_{\mathbb T}$ and $F_f=f$ on $\mathbb T$.
\end{proposition}

In connection to this proposition one ought to note that the functions
$\varPhi_1$ and $\varPhi_2$ separately do not belong to any of the spaces
$H^p(\mathbb D)$ as $p>0$, and only the function $F_f$ possesses the good
properties mentioned above.

In spite of the circumstance that the function $\psi$ from \eqref{eq-s:exm-2}
is never univalent in $\mathbb D$ (see, for instance, \cite{Pom1975book},
Section~5.4), Proposition~\ref{pro-s:emx-2} is an important ingredient of the
construction of domain $G$ from Proposition~\ref{pro-s:emx-1}. The desired
domain $G$ may be constructed as an image a certain Lipschits domain
$\Omega\subset\mathbb D$ under conformal mapping by univalent and Lipschitz
in $\Omega$ function $\varphi$, where $\Omega$ is obtained as a small
perturbation of $\mathbb D$, while $\varphi$ is obtained as a small
perturbation of $\psi$ with respect to the $L^2$-norm on the boundary.

The desired domain $\Omega$ is constructed as a kernel of a decreasing
sequence of Jordan domains $\Omega_j$, $j=0,1,2\ldots$, where
$\Omega_0=\mathbb D$, with uniform estimate of Lipschitz constants of their
boundaries. Let us describe the first step of this construction. The domain
$\Omega_1$ is obtained from $\Omega_0=\mathbb D$ using the celebrated
Privalov's ice-cream cone construction, see \cite{Koo1998book}, Chapter~III,
Section~D. Indeed, the $L^2$-norm of the function $\psi'-1$ on $\mathbb T$
can be made arbitrary small together with the quantity
$\sum_{k=k_0}^{\infty}k^{-2}$; the same takes place for $L^2$-norm of the
corresponding non-tangential maximal function. Thus, for a given $\delta>0$
and for sufficiently large $k_0=k_0(\delta)$ there exists a domain
$\Omega_1=\Omega_0\setminus\bigcup_{\alpha\in A}T_\alpha$, where $A$ is at
most countable set of indices and $T_\alpha$, $\alpha\in A$, are mutually
disjoint closed isosceles triangles with the bases on $\partial\Omega_0$ such
that the function $\psi'$ is continuous on $\overline\Omega_1$ and such that
everywhere on $\overline\Omega_1$ it holds $|\psi'(z)-1|<\delta$. Moreover,
the triangles $T_\alpha$, $\alpha\in A$ may be chosen such that the angles at
the base of every $T_\alpha$ are less than $\delta$, and the sum of
perimeters of all $T_\alpha$, $\alpha\in A$ is also less than $\delta$.

The part of the boundary $\partial\Omega_1$ that does not belong to
$\partial\Omega_0$ consists of at most countable family of intervals of the
total length $\ell<\delta$. Our aim is to modify the function $\psi$ on these
intervals. Take a finite family of pairwise disjoint closed segments of the
total length greater than $0.9\,\ell$ belonging to the intervals of this
family. For every such segment $I$ let us proceed as follows.

Let $Q$ be a circular lune whose boundary consists of $I$ and the circular
arc $I'$ of some circle with center lying outside $\Omega_1$ and with the
angle measure less than $\delta^2$ (the arcs $I$ and $I'$ intersect only by
their end-points). Take a function $\chi$ that maps conformally the exterior
of $Q$ onto $\mathbb D$ with the normalization $\chi(\infty)=0$ and assume
that $\chi$ is already extended to the homeomorphism of the corresponding
closed domains. Define
$$
\psi_I(z)=\sum_{k=k_0(I)}^{\infty}\frac{(\chi(z))^{m_k}}{km_k},
$$
where $k_0(I)$ is sufficiently large. At the next step we remove from
$\Omega_1$ all domains $Q$ constructed above and all isosceles triangles
constructed on all arcs $I'$ using the ice-cream cone construction as it was
mentioned above. The resulting domain will be $\Omega_2$. For this domain we
repeat the same construction as before.

It can be readily verified that the result of Proposition~\ref{pro-s:emx-2}
will be preserved if replace the function $\psi$ with $\psi(z)-z$. Therefore
we can use the function $\psi_I$ constructed above for modification of the
initial function $\psi$. Indeed, we can define $\varphi=\psi+\sum_I\psi_I$,
where the sum is taken over all $\psi_I$ constructed in all steps. It is not
difficult to prove that $\varphi$ is univalent in the resulting domain
$\Omega$ which is the kernel of the sequence $\Omega_j$, $j\geqslant0$.

Finally, we need to check that the domain $\Omega$ is
$\overline\partial^2$-regular. In order to verify this property we will use
the following criterion of $\overline\partial^2$-regularity, which was proved
in \cite{Maz2009sbm}, Lemma~4.2, and which is obtained using the functional
analysis methods and the Rudin--Carleson theorem stating that any compact set
of zero length is an interpolation peak set for continuous holomorphic
functions, see \cite{Rud1956pams} and \cite{Car1957mz}.

\begin{proposition}\label{pro-s:emx-3}
The domain $G$ is $\overline\partial^2$-regular if and only if the following
property is satisfied. There exists an increasing sequence $(E_n)$,
$n=1,2,\ldots$, of closed subsets of $\partial G$ such that

\smallskip
1\textup)~the length of the set $\partial G\setminus\bigcup_nE_n$ is zero,
and

\smallskip
2\textup)~for every $f\in C(\partial G)$ with $\|f\|_{\partial G}\leqslant1$,
for every $E_n$, and for every $\varepsilon>0$ there exists $F\in
C(\overline{G})\cap\mathcal O(G,\overline\partial^2)$ such that
$|F(z)|\leqslant2$ for $z\in G$ and $|f(z)-F(z)|<\varepsilon$ for $z\in E_n$.
\end{proposition}


\section{Uniform approximation by $\mathcal L$-analytic\\polynomials and
$\mathcal L$-Dirichlet problem} \label{s:approx}

Denote by $\mathcal P$ the class of all polynomials in the complex variable
$z$. Let $\mathcal L\in\Ell$, and let $\lambda_1$ and $\lambda_2$ be the
characteristic roots of $\mathcal L$. Recall that by $\mathcal L$-analytic
polynomial we mean any complex-valued polynomial $P$ in two real variables
that satisfies the equation $\mathcal LP=0$. If $\lambda_1\neq\lambda_2$,
then any $\mathcal L$-analytic polynomial has the form \eqref{eq-solrep1},
where $f_1,f_2\in\mathcal P$. If $\lambda_1=\lambda_2$, then any $\mathcal
L$-analytic polynomial is the function of the form \eqref{eq-solrep2}, where
also $f_0,f_1\in\mathcal P$. Denote by $\mathcal P_{\mathcal L}$ the class of
all $\mathcal L$-analytic polynomials.

Note that if we reduce the operator $\mathcal L$ to the form \eqref{eq-op2se}
with $\tau\in[0,1)$ or \eqref{eq-op2nse} with $\tau\in(0,1)$, respectively,
then $\mathcal L$-analytic polynomials will take the form
\eqref{eq-soltau-se} or \eqref{eq-soltau-nse}, respectively, where
$g,h\in\mathcal P$. Moreover, any $\mathcal L_0$-analytic polynomials has the
form $\overline{z}P_1(z)+P_0(z)$ with $P_0,P_1\in\mathcal P$ (since $\mathcal
L_0=\overline\partial{}^2$).

For a given compact set $X\subset\mathbb C$ we denote by $P_{\mathcal L}(X)$
the space of all functions that can be approximated uniformly on $X$ by
$\mathcal L$-analytic polynomials. In other words, a function $f$ belongs to
$P_{\mathcal L}(X)$ if and only if for every $\varepsilon>0$ there exists
$P\in\mathcal P_{\mathcal L}$ such that $\|f-P\|_X<\varepsilon$. It is clear
that
$$
P_{\mathcal L}(X)\subset C_{\mathcal L}(X):=C(X)\cap\mathcal O(X^\circ,\mathcal L),
$$
where $X^\circ=\Int(X)$ is the interior of $X$. Let us consider the following
problem.
\begin{problem}\label{prob-approx}
To describe compact sets $X\subset\mathbb C$ for which $P_{\mathcal
L}(X)=C_{\mathcal L}(X)$.
\end{problem}
More precisely, it is demanded in Problem 2 to obtain necessary and
sufficient conditions on $X$ in order that the equality $P_{\mathcal
L}(X)=C_{\mathcal L}(X)$ is satisfied. This problem is the well-known
classical problem in complex analysis. Its statement is traced to the
classical problems on uniform approximation by harmonic polynomials and by
polynomials in the complex variables that were solved by Walsh and Mergelyan,
respectively. One ought to emphasize that Problem~\ref{prob-approx} is still
open in the general case. We refer the interested reader to
\cite{MazParFed2012rms}, where one can find a detailed survey concerning the
matter. Let us also note that Problem~\ref{prob-approx} is closely related
with the problem on approximation of functions $f\in C_{\mathcal L}(X)$ by
functions which are $\mathcal L$-analytic in neighborhoods of $X$. This is a
classical problem in the case of holomorphic and harmonic functions and its
consideration is out of the scope of this paper. The studies of this problem
in the context of approximation by solutions of general elliptic equations
was started in 1980s--1990s. A more or less exhaustive bibliography on this
subject may be found in \cite{MazParFed2012rms}, but let us mention here a
couple of important works \cite{ParVer1994msand}, \cite{Tar1987sbm},
\cite{Ver1987duke}, \cite{Ver1993pjm} and \cite{Ver1994nato}.

The only case where Problem~\ref{prob-approx} was solved completely is the
case $\mathcal L=\Delta$, and, as a clear consequence, a slightly more
general case when $\mathcal L$ has real coefficients (up to a common complex
multiplier). To state the corresponding result we need the concept of a
Carath\'eodory compact set.

We recall, that a compact set $X\subset\mathbb C$ is called a
\emph{Carath\'eodory compact set}, if $\partial X=\partial\widehat{X}$, where
$\widehat{X}$ denotes the union of $X$ and all bounded connected components
of $\mathbb C\setminus X$.

\begin{atheorem}\label{thm-wl}
Let $X$ be a compact set in $\mathbb C$. Then $P_\Delta(X)=C_\Delta(X)$ if
and only if $X$ is a Carath\'eodory compact set.
\end{atheorem}

This remarkable result was proved by Walsh at the end of 1920s, see
\cite{Wal1929bams}, and nowadays it is called the Walsh--Lebesgue theorem in
view of the crucial role the Lebesgue theorem (Theorem~\ref{thm-lebesgue})
plays in the proof. Note that in \cite{Wal1929bams} only the case of nowhere
dense compact sets was considered explicitly, but the general case may be
obtained as a consequence of this partial result. Let us also observe that
the first, to the best of our knowledge, formulation of the Walsh--Lebesgue
theorem in the above form was presented in \cite{Par1994sbm}, Section~1. The
deep enough exposition of the Walsh--Lebesgue theorem and certain related
topics may be found in Chapter~2 of the book \cite{Gam1984book}.

The second case, when the substantial progress was achieved in studies of
Problem~\ref{prob-approx}, is the case where $\mathcal
L=\overline\partial{}^2$. In this case Problem~\ref{prob-approx} was solved
completely for Carath\'eodory compact sets. The following results was
obtained in \cite{CarParFed2002sbm}, Theorem~2.2.

\begin{atheorem}\label{thm-cfp}
Let $X$ be a Carath\'eodory compact set in $\mathbb C$. Then
$P_{\overline\partial^2}(X)=C_{\overline\partial^2}(X)$ if and only if any
bounded connected component of the set $\mathbb C\setminus X$ is not a
Nevanlinna domain.
\end{atheorem}

The concept of a Nevanlinna domain is the special analytic characteristic of
bounded simply connected domains in $\mathbb C$. It's formal definition is
given in \cite{Fed1996mn}, Definition~1, in the case of Jordan domains with
rectifiable boundaries, and in \cite{CarParFed2002sbm}, Definition~2.1, in
the general case. We are not going to define this concept explicitly here,
but we ought to note that the property of a given domain $G$ in $\mathbb C$
to be a Nevanlinna domain consists in the possibility of representing the
function $\overline{z}$ almost everywhere on $\partial G$ in the sense of
conformal mappings as a ratio of two bounded holomorphic functions in $G$.
The properties of Nevanlinna domains has been studied in detail during the
two last decades (see, for instance, \cite{Fed2006psim, BarFed2011sbm,
Maz2016spmj, BarFed2017jam, Maz2018spmj, BelFed2018rms, BelBorFed2019jfa}. It
was shown that the class of Nevanlinna domains is rather big in spite of the
fact that the definition of a Nevanlinna domain imposes quite rigid condition
to the boundary of the domain under consideration. For instance, there exists
such Nevanlinna domains $G$ that the Hausdorff dimension of $\partial G$
could take any value in $[1,2]$ (see \cite{BelBorFed2019jfa}, Theorem~3).

For compact sets $X$ which are not Carath\'eodory compact sets the question
whether the equality $P_{\overline\partial^2}(X)=C_{\overline\partial^2}(X)$
holds or not, is solved only for some particular cases. In the general case
the answer to this question is know only in the form of a certain
approximability condition of a reductive nature. For certain
non-Carath\'eodory compact sets of a special form the sufficient
approximability conditions were obtained in \cite{BoiGauPar2004izv},
\cite{CarParFed2002sbm} and \cite{CarFed2005otaa}. One interesting and
helpful tool using in these works is the concept of an analytic balayage of
measures which was introduced by D.~Khavinson \cite{DKha1988ca}, which was
rediscovered in \cite{CarParFed2002sbm} in a slightly different terms, and
which was studied by several authors both as an object of an independent
interest and in connection with properties of badly approximable functions in
$L^p(\mathbb T)$ (see \cite{AbaFed2018cras} and bibliography therein). It
seems to us interesting and appropriate to note this point.

In order to proceed further with our discussions of Problem~\ref{prob-approx}
and its relations with $\mathcal L$-Dirichlet problem let us introduce one
more space of functions. For a pair of compact sets $X$ and $Y$ in $\mathbb
C$ with $X\subseteq Y$, let $A_{\mathcal L}(X,Y)$ be the closure in $C(X)$ of
the space $\{f|_X\colon f\in\mathcal O(U_f(Y),\mathcal L)\}$, where $U_f(X)$
is some, depending on $f$, neighborhood of $Y$. The typical case when we are
needed this space, is the case when $X=\partial G$ and $Y=\overline{G}$ for
some bounded simply connected domain $G$ in $\mathbb C$. It is clear, that
$$
P_{\mathcal L}(X)\subset A_{\mathcal L}(X,Y)\subset A_{\mathcal L}(X,X)
\subset C_{\mathcal L}(X).
$$
Let us also note that if $X$ has a connected complement, then $A_{\mathcal
L}(X,Y)=P_{\mathcal L}(X)$ for every $Y$ with $X\subseteq Y$. This is a
direct consequence of the standard Runge's pole shifting method which remains
valid for $\mathcal L$-analytic functions, see \cite{Nar1973book},
Section~3.10.

Although a complete solution to Problem~\ref{prob-approx} has not yet been
obtained (for any operator $\mathcal L$ under consideration, except for the
Laplace operator $\Delta$ and for operators that can be reduced to it by a
not degenerate real-linear transformation of the plane), the following
approximability criterion of a reductive nature was established in connection
with this problem in the early 2000s.

\begin{atheorem}\label{thm-bgp+z}
Let $X$ be a compact set in $\mathbb C$ having disconnected complement, and
let $\mathcal L$ be an arbitrary operator of the form \eqref{eq-ellop}. The
equality $P_{\mathcal L}(X)=C_{\mathcal L}(X)$ takes place if and only if for
every connected component $G$ of the set $\Int(\widehat{X})$ such that
$G\cap(\mathbb C\setminus X)\neq\emptyset$ the equality is satisfied
$$
A_{\mathcal L}(\overline{G}\cap X,\overline{G})=C_{\mathcal L}(\overline{G}\cap X).
$$
\end{atheorem}

This theorem was firstly proved in \cite{BoiGauPar2004izv} for bianalytic
functions, and almost immediately after that the proof was modified for
general $\mathcal L$ in \cite{Zai2004izv}. Theorem~\ref{thm-bgp+z} allows us
to reduce the problem of $\mathcal L$-analytic polynomial approximation on a
given compact set $X$ to compact subsets of $X$ having more simple
topological structure.

In the particular case, when $X$ is a Carath\'eodory compact set,
Theorem~\ref{thm-bgp+z} gives that the equality $P_{\mathcal
L}(X)=C_{\mathcal L}(X)$ takes place if and only if for any bounded connected
component $G$ of the set $\mathbb C\setminus X$ one has $A_{\mathcal
L}(\partial G,\overline{G})=C(\partial G)$. In this case the domain $G$ under
consideration is such that $\partial G=\partial G_\infty$, where $G_\infty$
is the unbounded connected component of the set $\mathbb
C\setminus\overline{G}$. Such domain $G$ is called a \emph{Carath\'eodory
domain}. It is clear that every Carath\'eoodry domain is simply connected and
possesses the property $G=\Int(\overline{G})$.

Given a bounded simply connected domain $\Omega$ let us denote by
$\partial_a\Omega$ the accessible part of $\partial\Omega$, namely the set of
all points $\zeta\in\partial\Omega$ which are accessible from $G$ by some
Jordan curve lying in $G\cup\{\zeta\}$ and ending at $\zeta$.

We are going to present one criterion in order that the equality $A_{\mathcal
L}(\partial G,\overline{G})=C(\partial G)$ holds for a given Carath\'eodory
domain $G$ (see Theorem~\ref{thm-tz} below). This result was firstly obtained
in \cite{Zai2002mn}, but the proof given therein is not enough complete in a
certain place. To state Theorem~\ref{thm-tz} we need to introduce yet another
space of functions and give one more definition. Let $G$ be a Carath\'eodory
domain in $\mathbb C$, and let $\varphi$ be some conformal mapping from
$\mathbb D$ onto $G$. One says that a holomorphic function $f$ in $G$ belongs
to the space $\AC(G)$, if the function $f\circ\varphi$ is extendable to a
function which is continuous on $\overline{\mathbb D}$ and absolutely
continuous on $\mathbb T$. It follows from the standard facts about conformal
mappings, that for every function $f\in\AC(G)$, for every point
$\zeta\in\partial_aG$, and for every path $\varUpsilon$ lying in
$G\cup\{\zeta\}$ and ending at $\zeta$, the limit of $f$ along $\varUpsilon$
exists and is equal to the same value $f(\zeta)$, which is called a boundary
value of $f$ at $\zeta$.

\begin{definition}\label{dfn-lspec}
Let $\mathcal L$ be an operator of the form \eqref{eq-ellop} with
characteristic roots $\lambda_1\neq\lambda_2$. A Carath\'eodory domain $G$ is
called an $\mathcal L$-special domain, if there exist two functions
$F_1\in\AC(T_{(1)}G)$ and $F_2\in\AC(T_{(2)}G)$ such that for every
$\zeta\in\partial_aG$ one has $F_1(T_{(1)}\zeta)=F_2(T_{(2)}\zeta)$.
\end{definition}

The real linear transformations $T_{(1)}$ and $T_{(2)}$ of the plane are
defined just after the formula \eqref{eq-solrep1} in Section~\ref{s:aux}
above.

\begin{theorem}\label{thm-tz}
Let $G$ be a Carath\'eodory domain in $\mathbb C$ and let $\mathcal L$ be an
operator of the form \eqref{eq-ellop} with characteristic roots
$\lambda_1\neq\lambda_2$.

\smallskip
1\textup) The equality $A_{\mathcal L}(\partial G,\overline{G})=C(\partial
G)$ takes place if and only if the domain $G$ is not an $\mathcal L$-special
domain.

\smallskip
2\textup) If $\mathcal L\in\SE$ then every Carath\'eodory domain in $\mathbb
C$ is not $\mathcal L$-special.
\end{theorem}

\begin{proof}
Let $\varphi$ be a conformal mapping from $\mathbb D$ onto $G$ and let $\psi$
be the respective inverse mapping. Without lost of generality we may assume
that $\varphi$ has angular boundary value at the point $1$. By virtue of
\cite{Pom1992book}, Propositions~2.14 and~2.17, we have
$\partial_aG=\{\varphi(\xi)\colon \xi\in\mathcal F(\varphi)\}$, where
$\mathcal F(\varphi)$ is the Fatou set of $\varphi$, that is the set of all
points $\xi\in\mathbb T$, where $\varphi$ has finite angular boundary values
$\varphi(\xi)$ according to the classical Fatou's theorem. As it was shown in
\cite{CarFed2005otaa}, $\partial_aG$ is a Borel set. In view of
\cite{CarFed2005otaa}, Corollary~1, the functions $\varphi$ and $\psi$ can be
extended to Borel measurable functions (denoted also by $\varphi$ and $\psi$)
on $\mathbb D\cup\mathcal F(\varphi)$ and $G\cup\partial_aG$ respectively in
such a way, that $\varphi(\psi(\zeta))=\zeta$ for all $\zeta\in\partial_aG$
and $\psi(\varphi(\xi))=\xi$ for all $\xi\in\mathcal F(\varphi)$.

Let $\omega$ be the measure on $\partial G$ defined by
$\omega:=\varphi(d\xi)$ (see \cite{CarFed2005otaa}, Section~3, for detailed
construction of this measure and its properties). In fact $\omega$ is a
measure on $\partial_aG$ and has no atoms. Moreover,
$|\omega(\cdot)|=2\pi\omega(\varphi(0),\cdot,G)$, where
$\omega(\varphi(0),\cdot,G)$ is the harmonic measure on $\partial G$
evaluated with respect to $\varphi(0)$ and $G$.

Taking into account Remark~\ref{rem-solspace} we assume that $\mathcal L$ is
already reduced to the form $c\mathcal L^\dag_\tau$, $\tau\in[0,1)$ in the
case where $\mathcal L$ is strongly elliptic, or to the form $c\mathcal
L_\tau$, $\tau\in(0,1)$, in the opposite case. Since the further
constructions are actually the same in both these cases, we will deal in
details only with the case $\mathcal L=\mathcal L_\tau$. Let us recall that
$z_\tau=z-\tau\overline{z}$ and $T\colon z\mapsto z_\tau$.

Suppose that $G$ is such that $A_\tau(\partial G,\overline{G}):=A_{\mathcal
L_\tau}(\partial G,\overline{G})\neq C(\partial G)$. It means that there
exists a (finite complex-valued Borel) measure $\mu$ on $\partial G$ which is
orthogonal to the space $A_\tau(\partial G,\overline{G})$. In view of
\eqref{eq-soltau-nse} the orthogonality of $\mu$ to $A_\tau(\partial
G,\overline{G})$ means that $\mu$ is orthogonal to the space
$R(\overline{G})$ consisting of functions which can be approximated uniformly
on $\overline{G}$ by rational functions in the complex variable with poles
lying outside $\overline{G}$ and to the space $\{h(z_\tau)\colon h\in
R(T\overline{G})\}$. The orthogonality of $\mu$ to $R(\overline{G})$ means
that $\int f\,d\mu=0$ for every $f\in R(\overline{G})$.

Since $\mu$ is orthogonal to $R(\overline{G})$, then, according to
\cite{CarFed2005otaa}, Theorem~2, there exists a function $b$ belonging to
the Hardy space $H^1(\mathbb D)$ in the unit disk such that
$$
\mu=(b\circ\psi)\,\omega.
$$
Let $\varPhi$ be some primitive to $b$ in $\mathbb D$. According to
\cite{Pri1950book}, Section~II.5.7, $\varPhi\in C(\overline{\mathbb D})$ and
$\varPhi$ is absolutely continuous on $\mathbb T$. Moreover, for any point
$\xi\in\mathbb T$ we have $\varPhi(\xi)=\varPhi(1)+\nu(\varUpsilon_\xi)$,
where $\nu$ is the measure on $\mathbb T$ defined by the setting
$d\nu=b\,d\xi$, and $\varUpsilon_\xi$ is the arc of $\mathbb T$ running from
$1$ to $\xi$ in the positive direction. Finally, let $F(z)=\varPhi(\psi(z))$,
$z\in G$. Since $G$ is a Carath\'eodory domain, for any point
$\zeta\in\partial_aG$ there exists a unique point $\xi\in\mathbb T$ such that
$\xi\in\mathcal F(\varphi)$ and $\varphi(\xi)=\zeta$, see
\cite{CarFed2005otaa}, Proposition~1. Therefore, $F$ is well-defined on
$\partial_aG$ and $F\in\AC(G)$.

Next we do the same thing for the domain $G_\tau:=TG$, which is a
Carath\'eodory one, and for which it holds $\partial_aG_\tau=T(\partial_aG)$.
Let $\varphi_\tau$ be some conformal mapping from $\mathbb D$ onto $G_\tau$
such that $1\in\mathcal F(\varphi_\tau)$ and $T\varphi(1)=\varphi_\tau(1)$.
As previously we denote by $\psi_\tau$ the inverse mapping for $\varphi_\tau$
in $G_\tau$ and by $\omega_\tau$ the measure $\varphi_\tau(d\xi)$ on
$\partial G_\tau$. Furthermore, let $\mu_\tau$ be the measure on $\partial
G_\tau$ defined by the setting $\mu_\tau(E)=\mu(T^{-1}E)$, where $E$ is a
Borel set. It is clear that $\mu_\tau$ is orthogonal to
$R(\overline{G_\tau})$.

Repeating the construction of $b$, $\varPhi$, $\nu$ and $F$ given above,
using $\mu_\tau$, $\varphi_\tau$ and $\psi_\tau$ instead of $\mu$, $\varphi$,
and $\psi$, respectively, we obtain the function $b_\tau\in H^1(\mathbb D)$
such that $\mu_\tau=(b_\tau\circ\psi_\tau)\,\omega_\tau$, and take
$\varPhi_\tau$ to be the primitive of $b_\tau$ such that
$\varPhi_\tau(1)=\varPhi(1)$. We have used here the fact that
$\varPhi_\tau\in C(\overline{\mathbb D})$. Finally we put
$F_\tau(w)=\varPhi_\tau(\psi_\tau(w))$, $w\in G_\tau$, so that
$F_\tau\in\AC(G_\tau)$.

It remains to show, how the functions $F$ and $F_\tau$ are related to each
other. Let $\varGamma$ is a closed Jordan curve and let $\varOmega$ be the
domain bounded by $\varGamma$. Take three points
$\zeta_1,\zeta_2,\zeta_3\in\varGamma$. One says that a triplet
$(\zeta_1,\zeta_2,\zeta_3)$ is positive with respect to $\varOmega$, if
$\zeta_2$ lies on the arc of $\varGamma$ running from $\zeta_1$ to $\zeta_3$
in the positive direction on $\varGamma$ (with respect to $\varOmega$).

Let $\alpha,\xi\in\mathbb T\cap\mathcal F(\varphi)$ are such that the triplet
$(1,\alpha,\xi)$ is positive with respect to $\mathbb D$. Let
$\alpha'=\psi_\tau(T\varphi(\alpha))$ and $\xi'=\psi_\tau(T\varphi(\xi))$.

We claim that the triplet $(1,\alpha',\xi')$ is also positive with respect to
$\mathbb D$. Assuming this claim already proved, let us finish the proof of
Theorem~\ref{thm-tz}. Since the triplets $(1,\alpha,\xi)$ and
$(1,\alpha',\xi')$ are both positive with respect to $\mathbb D$, and since
the sets $\mathcal F(\varphi)$ and $\mathcal F(\varphi_\tau)$ are everywhere
dense in $\mathbb T$, then for the arc $\varUpsilon_\xi$ defined above we
have
\begin{equation}\label{eq-x1}
\psi_\tau(T\varphi(\varUpsilon_\xi\cap\mathcal F(\varphi)))=
\varUpsilon_{\xi'}\cap\mathcal F(\varphi_\tau).
\end{equation}
Using this equality we have that for $\zeta=\varphi(\xi)\in\partial_aG$ and
$\xi'=\psi_\tau(T\zeta)$ it holds
\begin{align*}
F_\tau(T\zeta)-\varPhi_\tau(1)& =\varPhi_\tau(\xi')-\varPhi_\tau(1)= %
\nu_\tau(\varUpsilon_{\xi'}\cap\mathcal F(\varphi_\tau))\\
&=\mu_\tau(\varphi_\tau(\varUpsilon_{\xi'}\cap\mathcal F(\varphi_\tau)))= %
\mu_\tau(T\varphi(\varUpsilon_\xi\cap\mathcal F(\varphi)))\\
& = \mu(\varphi(\varUpsilon_\xi\cap\mathcal F(\varphi)))
= \nu(\varUpsilon_\xi\cap\mathcal F(\varphi)) = \varPhi(\xi)-\varPhi(1)\\
&= F(\zeta)-\varPhi(1).
\end{align*}
Since $\varPhi_\tau(1)=\varPhi(1)$ we finally have that
$F_\tau(T\zeta)=F(\zeta)$ for every $\zeta\in\partial_aG$, as demanded in
Definition~\ref{dfn-lspec} and Theorem~\ref{thm-tz}.

It remains to prove our claim stating that the triplet $(1,\alpha',\xi')$ is
positive whenever the initial triplet $(1,\alpha,\xi)$ is positive (with
respect to $\mathbb D$).

For two points $p,q\in\mathbb T$ let $C_{p,q}$ be some circular arc belonging
to $\mathbb D\cup\{p,q\}$ that starts at $p$, ends at $q$ and intersects
$\mathbb T$ at both $p$ and $q$ non-tangentially. Let $\Omega_{1,\alpha,\xi}$
be the Jordan domain bounded by the closed Jordan curve
$\varGamma_{1,\alpha,\xi}=C_{1,\alpha}\cup C_{\alpha,\xi}\cup C_{\xi,1}$,
where we assume that $C_{1,\alpha}$, $C_{\alpha,\xi}$ and $C_{\xi,1}$ are
taken in such a way that $\varGamma_{1,\alpha,\xi}$ is a closed Jordan curve,
and $0\in\Omega_{1,\alpha,\xi}$. It is clear that the triplet
$(1,\alpha,\xi)$ is positive with respect to $\Omega_{1,\alpha,\xi}$.

Since $\alpha$ and $\xi$ are Fatou points for $\varphi$, and since all three
arcs $C_{1,\alpha}$, $C_{\alpha,\xi}$ and $C_{\xi,1}$ intersects $\mathbb T$
non-tangentially, we have $\varphi\in C(\overline\Omega_{1,\alpha,\xi})$.
Moreover, since $G$ is a Carath\'eodory domain, then for each point
$\zeta\in\partial_aG$ there exists a unique point $\eta\in\mathcal
F(\varphi)$ such that $\varphi(\eta)=\zeta$ (see \cite{CarFed2005otaa},
Proposition~1), and hence $\varphi$ is injective on
$\overline\Omega_{1,\alpha,\xi}$. Finally, in view of the Carath\'eodory
extension theorem (see, for instance, \cite{Pom1992book}, Theorem~2.6) the
domain $\varphi(\Omega_{1,\alpha,\xi})\subset G$ is a Jordan domain.
Furthermore, the triplet $(\varphi(1),\varphi(\alpha),\varphi(\xi))$ is
positive with respect to $\varphi(\Omega_{1,\alpha,\xi})$ (it can be readily
verified by direct computation of index of the curve
$\varphi(\varGamma_{1,\alpha,\xi}$ with respect to the point
$\varphi(0)\in\varphi(\Omega_{1,\alpha,\xi})$). Since $T$ is sense-preserving
mapping, then the triplet $(\varphi_\tau(1),\varphi(\alpha'),\varphi(\xi'))=
(T\varphi(1),T\varphi(\alpha),T\varphi(\xi))$ is positive with respect to the
domain $T\varphi(\Omega_{1,\alpha,\xi})$.

Let us consider the domain
$\Omega'=\psi_\tau(T\varphi(\Omega_{1,\alpha,\xi}))\subset\mathbb D$. It is
clear that $\partial\Omega'\subset\mathbb D\cup\{1,\alpha',\xi'\}$. Repeating
the arguments used above we conclude that $\Omega'$ is a Jordan domain and
$\varphi_\tau$ is continuous and injective on $\overline{\Omega'}$. Therefore
the triplet $(1,\alpha',\xi')$ is positive with respect to $\mathbb D$, as it
was claimed.

Let us briefly explain how to proceed in the remaining case, namely in the
case when $\mathcal L$ is reduced to the form $c\mathcal L^\dag_\tau$. For a
given set $E\subset\mathbb C$ we put $E_C=\{\overline{z}\colon z\in E\}$.
Similarly to the previous case we consider some measure $\mu$ on $\partial G$
orthogonal to $A_{\mathcal L^\dag_\tau}(\partial G,\overline{G})$. In view of
\eqref{eq-soltau-se} it means that $\mu$ is orthogonal to the space
$\{g(\overline{z})\colon g\in R(\overline{G_C})\}$ and to the space
$\{h(z_\tau)\colon h\in R(\overline{G_\tau})\}$. The first orthogonality
condition yields that the measure $\mu_C$ defined by the setting
$\mu_C(E)=\mu(E_C)$ on $\partial G_C$ is orthogonal to $R(\overline{G_C})$.

Now we can repeat the constructions of $F$ and $F_\tau$ given above using
$\mu_C$ instead of $\mu$ and keeping $\mu_\tau$ unchanged. Doing this we
obtain the functions $F\in\AC(G_C)$ and $F_\tau\in\AC(G_\tau)$ such that
$-F(\overline{z})=F_\tau(z_\tau)$ for all $z\in\partial_aG$. The negative
sign at the left-hand side of this equality is related with the following
circumstance. Let $(1,\alpha,\xi)$ be a positive (with respect to $\mathbb
D$) triplet of points in $\mathbb T\cap\mathcal F(\varphi)$, let $\varphi_C$
be some conformal mapping from $\mathbb D$ onto $G_C$ such that $1\in\mathcal
F(\widetilde\varphi)$ and $C\varphi(1)=\widetilde\varphi(1)$, and let
$\psi_C=\varphi_C^{-1}$. Then the triplet $(1,\alpha'',\xi'')$, where
$\alpha''=\psi_C(\overline{\varphi(\alpha)})$ and
$\xi''=\psi_C(\overline{\varphi(\xi)})$, is a negative triplet (since the
mapping $z\mapsto\overline{z}$ reverses orientation), and hence
$\psi_C(\varphi(\varUpsilon_\xi\cap\mathcal F(\varphi))_C)= \mathbb
T\setminus(\varUpsilon_{\xi''}\cap\mathcal F(\varphi_\tau))$. It remains to
note that $\mu_C(\partial_aG_C)=0$ in view of orthogonality of $\mu_C$ to
constants.

To prove the second statement let us observe that the function $\varPhi$ can
be chosen in such a way that is has zeros on $\mathbb D$, but it has no zeros
on $\mathbb T$ (it is enough to chose a suitable value of $\varPhi(1)$). It
can be readily verified (see \cite{Zai2002mn}, the proof of Corollary~1, for
details), that the function $1/F(\overline{z})$ in $\mu$-integrable over
$\partial_aG$ and
$$
\int_{\partial_aG}\dfrac{d\mu(z)}{F(\overline{z})}=
\int_{\partial_aG_C}\dfrac{d\mu_C(w)}{F(w)}=
\int_{\mathbb T}\dfrac{d\varPhi(\xi)}{\varPhi(\xi)}=
\frac1{2\pi}\Delta_{\mathbb T}\Arg(\varPhi)>0,
$$
according to out assumption on $\varPhi$. On the other hand,
$$
\int_{\partial_aG}\dfrac{d\mu(z)}{F(\overline{z})}=
-\int_{\partial_aG_\tau}\dfrac{d\mu_\tau(w)}{F_\tau(w)}=
-\int_{\mathbb T}\dfrac{d\varPhi_\tau(\xi)}{\varPhi_\tau(\xi)}=
-\frac1{2\pi}\Delta_{\mathbb T}\Arg(\varPhi_\tau)\leqslant0,
$$
which is a clear contradiction. Therefore, there are no $\mathcal L$-special
domains in the strongly elliptic case. The proof is completed.
\end{proof}

In the case where $\mathcal L\in\NSE$ the concept of a $\mathcal L$-special
domain is quite poorly studied. The important fact is that $\mathcal
L$-special domains exist for any such $\mathcal L$, but only a few explicit
examples of such domains are known. Dealing with the concept of a $\mathcal
L$-special domain, let us refer to \cite{Zai2004izv} where some simple
statements are obtained that allow one to conclude that a given domain with
certain peculiar properties of the boundary is not $\mathcal L$-special for
every $\mathcal L\in\NSE$.

The following result is the direct corollary of Theorems~\ref{thm-bgp+z}
and~\ref{thm-tz}.

\begin{theorem}
Let $X\subset\mathbb C$ be a Carath\'eodory compact set, and $\mathcal
L\in\SE$. Then $P_{\mathcal L}(X)=C_{\mathcal L}(X)$.
\end{theorem}

It means that the sufficient approximability condition similar to the one
stated in the Walsh--Lebesgue theorem remains valid for general strongly
elliptic second order operators. The question whether this sufficient
approximability condition is also a necessary one in the case of general
$\mathcal L\in\SE$ is still open. The following conjecture which was posed in
\cite{ParFed1999sbm}, Conjecture~4.1~(2), and which is still open in the
general case, asserts that the proclaimed result is true.
\begin{conjecture}\label{con-approx-strongell}
Let $\mathcal L\in\SE$, and let $X$ be a compact set in $\mathbb C$. Then
$P_{\mathcal L}(X)=C_{\mathcal L}(X)$ if and only if $X$ is a Carath\'eodory
compact set.
\end{conjecture}

Note, that the statement of this conjecture has sense for any operator
$\mathcal L\in\NSE$, but the corresponding result is certainly failed.
Indeed, for every such $\mathcal L$ one can find a compact set $X$ (the union
of the some ellipse and its center) which is not a Carath\'eodory compact
set, but $C(X)=P_{\mathcal L}(X)$, see \cite{ParFed1999sbm}, Section~4. This
is related to the lack of solvability and the non-uniqueness of solutions of
the $\mathcal L$-Dirichlet problem in the corresponding ellipse.

The inverse statement in Conjecture~\ref{con-approx-strongell} is very
interesting open question. Of course it has an affirmative answer for every
operator $\mathcal L$ with complex conjugate characteristic roots (every such
case can be reduced to the harmonic one by means of suitable non degenerate
real-linear transformation of the plane). But in the general case it is
rather incomprehensible how to proof this result until it would be proved the
following quite plausible conjecture.
\begin{conjecture}\label{con-strongell}
For every $\mathcal L\in\SE$ any bounded simply connected domain
$G\subset\mathbb C$ is $\mathcal L$-regular.
\end{conjecture}
By Theorem~\ref{thm-vervog} this conjecture is true for Jordan domains
bounded by sufficiently regular curves, but it is not enough to prove
Conjecture~\ref{con-approx-strongell} in its full generality.

Let $\mathcal L\in\Ell$. We are going now to discuss the question on whether
the property of $\mathcal L$-regularity of a given domain $G\subset\mathbb C$
and the uniqueness property in $\mathcal L$-Dirichlet problem in $G$ can be
fulfilled simultaneously. We pay attention to this question because of its
connection with the weak maximum modulus principle for $\mathcal L$-analytic
functions. There are several different concepts referred as a weak maximum
modulus principle. We are dealing with the following one.

\begin{definition}
One says that a bounded simply connected domain $G\subset\mathbb C$ satisfies
the weak maximum modulus principle for $\mathcal L$, if there exists a number
$C(G,\mathcal L)>0$ such that for every function $f\in
C(\overline{G})\cap\mathcal O(G,\mathcal L)$ the inequality is satisfied
$$
\max\limits_{z\in\overline{G}}|f(z)|\leqslant C(\mathcal
L,G)\max\limits_{z\in\partial G}|f(z)|.
$$
\end{definition}

As an immediate consequence of the open mapping theorem one can see, that for
any domain $G$ which is $\mathcal L$-regular and possesses the uniqueness
property for $\mathcal L$-Dirichlet problem, the weak maximum modulus
principle for $\mathcal L$ is also took place. As usual, one says that $G$
possesses the uniqueness property for $\mathcal L$-Dirichlet problem, if for
every $h\in C(\partial G)$ the function $f\in C(\overline{G})\cap
O(G,\mathcal L)$ with $f|_{\partial G}=h$ is uniquely determined.

Furthermore, in many instances the weak maximum modulus principle for
$\mathcal L$ is an important tool to prove the property of $\mathcal
L$-regularity of $G$. Thus, in \cite{VerVog1997tams} the result stated in
Theorem~\ref{thm-vervog} above was proved using a certain version of the weak
maximum modulus principle for $\mathcal L$ in the domain under consideration,
see \cite{VerVog1997tams}, Theorem~7.3.

We are going to explain that for every $\mathcal L\in\NSE$ the weak maximum
modulus principle for $\mathcal L$ is certainly failed in a sufficiently wide
range of domains. This fact is a consequence of the results about uniform
$\mathcal L$-analytic polynomial approximation, which are considered here.

Let us start with the simple case when we do not need any special
approximation results to analyze the situation. Let $\mathcal
L=\overline\partial^2$ and let $G$ be an arbitrary bounded simply connected
domain in $\mathbb C$. Assume that $G$ is $\overline\partial^2$-regular and,
simultaneously, possesses the uniqueness property for
$\overline\partial^2$-Dirichlet problem. Take a function
$h=\dfrac1{z-a}\Big|_{\partial G}$ with $a\in G$ and consider the function
$f\in C(\overline{G})\cap\mathcal O(G,\overline\partial^2)$ such that
$f|_{\partial G}=h$. Then the function $f(z)(z-a)-1$ is also belonging to
$C(\overline{G})\cap\mathcal O(G,\overline\partial^2)$ and vanishes on
$\partial G$. Then, in view of the uniqueness property, $f(z)(z-a)-1=0$
identically in $G$, but it is failed, for instance, at the point $a$.

The given arguments in the bianalytic case are very short and simple, but
they seems to be not appropriate in a more general case. Indeed, the class
$\mathcal O(G,\overline\partial^2)$ of bianalytic functions in $G$ has a
structure of a holomorphic module (we are able to multiply bianalytic
functions to holomorphic ones keeping the class), but it is not the case for
more general $\mathcal L$.

In the case of general operator $\mathcal L\in\NSE$ we will use a different
construction. It is based on the following lemma.

\begin{lemma}\label{pro-Zaitsev}
Let $\mathcal L\in\NSE$, and let $G$ be a Carath\'eodory domain in $\mathbb
C$ such that $A_{\mathcal L}(\partial G,\overline{G})=C(\partial G)$. Then
for every point $z_0\in G$ and for $Y=\partial G\cup\{z_0\}$ it holds
$A_{\mathcal L}(Y,\overline{G})=C(Y)$.
\end{lemma}

This lemma is a particular case of \cite{Zai2006psim}, Theorem~3. Moreover,
according to Theorem~\ref{thm-bgp+z} this lemma yields the following more
general result (see \cite{Zai2006psim}, Theorem~3): if $X$ is a
Carath\'eodory compact set with disconnected complement and such that
$P_{\mathcal L}(X)=C_{\mathcal L}(X)$. Then for every point $z_0$ belonging
to any bounded connected component of $\mathbb C\setminus X$ it holds
$P_{\mathcal L}(X\cup\{z_0\})=C_{\mathcal L}(X\cup\{z_0\})$.

\begin{proof}[Proof of Lemma~\ref{pro-Zaitsev}.]
The proof of this lemma can be extract from the proof of \cite{Zai2006psim},
Theorem~3, but we present here a new fairly short proof, which can be
regarded as a somewhat modified version of the proof given in
\cite{Zai2006psim}. From the very beginning we assume that $\mathcal L$ is
already reduced to the form $\mathcal L_\tau$, $\tau\in(0,1)$. In what
follows we will use the same system of notations that in the proof of
Theorem~\ref{thm-tz}.

Arguing by contradiction, let us suppose that
$A_\tau(Y,\overline{G})=A_{\mathcal L_\tau}(Y,\overline{G})\neq C(Y)$. It
means that there exists a finite complex-valued Borel measure $\eta$ on $Y$
that is orthogonal to $A_{\tau}(Y,\overline{G})$, so that $\eta$ is
orthogonal to $R(\overline{G})$ and $\eta$ is orthogonal to
$\{h(z_\tau)\colon h\in R(T\overline{G})\}$. Notice, that the measure $\eta$
supported on $\partial G\cup\{z_0\}$ and $\eta_0=\eta(\{z_0\})\neq0$,
otherwise $\eta$ is the measure on $\partial G$ which is impossible since
$A_{\mathcal L}(\partial G,\overline{G})=C(\partial G)$.

Let $\varphi$ be some conformal mapping from $\mathbb D$ onto $G$ such that
$\varphi(0)=z_0$ and let $\omega_0=\varphi(d\xi/(2\pi))$, that is $\omega_0$
is the harmonic measure on $\partial G$ evaluated with respect to $G$ and
$z_0$. It is easy to check that the measure
$\eta_0\omega_0-\eta_0\delta_{z_0}$, where $\delta_{z_0}$ is the unit point
mass measure supported at the point $z_0$, is orthogonal to
$R(\overline{G})$. Next, let $\mu=\eta|_{\partial G}$. Therefore, the measure
$\eta+\eta_0\omega_0-\eta_0\delta_{z_0}=\mu+\eta_0\omega_0$ is also
orthogonal to $R(\overline{G})$. But this measure is supported on $\partial
G$ and hence $\mu+\eta_0\omega_0=(b\circ\psi)\,\omega_0$, where $b\in
H^1(\mathbb D)$, $b(0)=0$, and $\psi=\varphi^{-1}$. Finally, we have
$\mu=(b\circ\psi+\eta_0)\,\omega_0$. Let $\varPhi$ be some primitive of $b$
in $\mathbb D$. As it was mentioned in the proof of Theorem~\ref{thm-tz}
above, $\varPhi\in C(\overline{\mathbb D})$ and $\varPhi$ is absolutely
continuous on $\mathbb T$.

Let $\eta_\tau$ be the measure defined as follows
$\eta_\tau(E)=\eta(T^{-1}E)$. Using this measure and taking into account the
fact that it is orthogonal to $R(\overline{G_\tau})$ we can find $b_\tau\in
H^1(\mathbb D)$ such that the measure $\mu_\tau=\eta_\tau|_{\partial TG}$ has
the form $\mu_\tau=(b_\tau\circ\psi_\tau+\eta_0)\,\omega_{\tau,0}$, where
$\varphi_\tau$ is some conformal map from $\mathbb D$ onto $G_\tau$ such that
$\varphi_\tau(0)=Tz_0$, $\psi_\tau=\varphi_\tau^{-1}$, and
$\omega_{\tau,0}\varphi_\tau(d\xi/(2\pi))$ is the harmonic measure on
$\partial G_\tau$ evaluated with respect to $G_\tau$ and $Tz_0$. Let
$\varPhi_\tau$ be some primitive for $b_\tau$ in $\mathbb T$ such that
$\varPhi_\tau(1)=\varPhi(1)$.

Using \eqref{eq-x1} by the same way as in the proof of Theorem~\ref{thm-tz}
we obtain, that if $\xi=e^{i\vartheta}$ and
$\xi'=e^{i\vartheta'}=\psi_\tau(T\varphi(\xi))$, then
$$
\varPhi(e^{i\vartheta})+\frac{\eta_0}{2\pi{}i}i\vartheta=
\varPhi_\tau(e^{i\vartheta'})+\frac{\eta_0}{2\pi{}i}i\vartheta',
$$
which gives $\xi\exp(\varPhi(\xi))=\xi'\exp(\varPhi_\tau(\xi'))$. Putting
$F(z)=\varPhi(\psi(z))$, $z\in G$ and $F_\tau(z)=\varPhi_\tau(\psi_\tau(z))$,
$z\in G_\tau$ we have $F\in\AC(G)$, $F_\tau\in\AC(G_\tau)$ and
$$
\psi(z)e^{F(z)}=\psi_\tau(z_\tau)e^{F_\tau(z_\tau)},\quad z\in\partial_aG.
$$
It remains to observe that $\psi\,e^F\in\AC(G)$ and
$\psi_\tau\,e^{F_\tau}\in\AC(G_\tau)$, and both these functions are
non-constant. Therefore $G$ is $\mathcal L$-special domain and hence
$A_\tau(\partial G,\overline{G})\neq C(\partial G)$ which is a contradiction.
Therefore, $A_\tau(Y,\overline{G})=C(Y)$, as it is demanded.
\end{proof}

Using this lemma we are able to state the following result which was firstly
obtained in \cite{Zai2006psim}, Corollary~4 and which says that the weak
maximum modulus principle is lack for every operator $\mathcal L\in\NSE$ in
every Carath\'eodory domain.

\begin{theorem}
Let $G$ be a Carath\'eodory domain in $\mathbb C$, and let $\mathcal
L\in\NSE$. Then $G$ does not satisfy the weak maximum modulus principle for
$\mathcal L$.
\end{theorem}

We are not going to give a detailed proof of this theorem, but we present
here a more or less complete scheme of the proof for the sake of completeness
and for the reader convenience. As previously let us assume that $\mathcal
L=\mathcal L_\tau$, $\tau\in(0,1)$. Let $G$ is such that $A_\tau(\partial
G,\overline{G})=C(\partial G)$. Then, according to Lemma~\ref{pro-Zaitsev},
the function that equals to $1$ at the given point $z_0\in G$ and that
vanishes in $\partial G$ can be approximated uniformly on $G\cup\{z_0\}$ (and
hence on $\partial G$) with an arbitrary accuracy by function $\mathcal
L_\tau$-analytic in a neighborhood of $\overline{G}$. This function gives a
clear contradiction with the weak maximum modulus principle for $\mathcal
L_\tau$-analytic functions in $G$. Next, let $G$ is such that
$A_\tau(\partial G,\overline{G})\neq C(\partial G)$. It means that $G$ is a
$\mathcal L_\tau$-special domain. By definition it means that there exists
two non-constant functions $F\in\AC(G)$ and $F_\tau\in \AC(TG)$ such that
$F(z)=F_\tau(z_\tau)$ for every $z\in\partial_aG$. Let $\varphi$ be some
conformal mapping from $\mathbb D$ onto $G$. It is not difficult to show that
the function $F(\varphi(w))-F_\tau(\varphi(w)-\tau\overline{\varphi(w)})$,
$w\in\mathbb D$, is extended continuously to $\overline{\mathbb D}$ and
vanishes everywhere on $\mathbb T$. Therefore, the function
$F(z')-F_\tau(z'_\tau)\to0$ when $z'\in G$ tends to an arbitrary point
$z\in\partial_aG$. Therefore, the uniqueness property for the $\mathcal
L$-Dirichlet problem in $G$ fails and hence, the weak maximum modulus
principle is also fails for $\mathcal L$ in $G$.


\begin{thebibliography}{999}

\bibitem{AbaFed2018cras} %
E.~Abakumov, K.~Fedorovskiy, \textit{Analytic balayage of measures,
Carath\'eodory domains, and badly approximable functions in~$L^p$}, C.~R.
Math. Acad. Sci. Paris, \textbf{356} (2018), no.~8, 870--874.

\bibitem{BagFed2017psim}%
A.~O.~Bagapsh, K.~Yu.~Fedorovskiy, \textit{Uniform and $C^1$-approximation of
functions by solutions of second order elliptic systems on compact sets in
$\mathbb R^2$}, Proc Steklov Inst Math., \textbf{298} (2017), 35--50.

\bibitem{BarFed2011sbm}%
A.~D.~Baranov, K.~Yu.~Fedorovskiy, \textit{Boundary regularity of Nevanlinna
domains and univalent functions in model subspaces}, Sb. Math., \textbf{202}
(2011), no.~12, 1723--1740.

\bibitem{BarFed2017jam}%
A.~D.~Baranov, K.~Yu.~Fedorovskiy, \textit{On $L^1$-estimates of derivatives
of univalent rational functions}, J.~Anal. Math., \textbf{132} (2017),
63--80.

\bibitem{BelBorFed2019jfa}%
Yu.~Belov, A.~Borichev, K.~Fedorovskiy, \textit{Nevanlinna domains with large
boundaries}, J.~Funct. Anal., \textbf{277} (2019), 2617--2643.

\bibitem{BelFed2018rms}%
Yu.~S.~Belov, K.~Yu.~Fedorovskiy, \textit{Model spaces containing univalent
functions}, Russian Math. Surveys, \textbf{73} (2018), no.~1, 172--174.

\bibitem{BoiGauPar2004izv}%
A.~Boivin, P.~M.~Gauthier, P.~V.~Paramonov, \textit{On uniform approximation
by n-analytic functions on closed sets in $\mathbb C$}, Izv. Math.,
\textbf{68} (2004), no.~3, 447--459.

\bibitem{Car1957mz}%
L.~Carleson, \textit{Representation of continuous functions}, Math. Zeit.,
\textbf{66} (1957), 447--451.

\bibitem{Car1985jat}%
J.~J.~Carmona, \textit{Mergelyan's approximation theorem for rational
modules}, J.~Approx. Theory, \textbf{44} (1985), 113--126.

\bibitem{CarFed2005otaa}%
J.~J.~Carmona, K.~Yu.~Fedorovskiy, \textit{Conformal maps and uniform
approximation by polyanalytic functions}, Selected topics in complex
analysis, Oper. Theory Adv. Appl., \textbf{158}, Birkh\"auser, Basel, 2005,
109--130.

\bibitem{CarParFed2002sbm}%
J.~J.~Carmona, P.~V.~Paramonov, K.~Yu.~Fedorovskiy, \textit{On uniform
approximation by polyanalytic polynomials and the Dirichlet problem for
bianalytic functions}, Sb.~Math., \textbf{193} (2002), no.~10, 1469--1492.

\bibitem{Dav1974book}%
P.~Davis, \textit{The Schwarz function and its applications}, Carus Math.\
Monogr., \textbf{17}, Math.\ Ass.\ of America, Buffalo, NY 1974.

\bibitem{Fed1996mn}%
K.~Yu.~Fedorovskii, \textit{Uniform $n$-analytic polynomial approximations of
functions on rectifiable contours in $\mathbb C$}, Math. Notes, \textbf{59}
(1996), no.~4, 435--439.

\bibitem{Fed2006psim}%
K.~Yu.~Fedorovskii, \textit{On some properties and examples of Nevanlinna
domains}, Proc. Steklov Inst. Math., \textbf{253} (2006), 186--194.

\bibitem{Gam1984book}%
T.~W.~Gamelin, \textit{Uniform Algebras}, Chelsea Publishing Company, New
York, 1984.

\bibitem{DKha1988ca} %
D.~Khavinson, \textit{F.~and M.~Riesz theorem, analytic balayage, and
problems in rational approximation}, Constr. Approx., \textbf{4} (1988),
no.~4, 341--356.

\bibitem{Koo1998book}%
P.~Koosis, \textit{Introduction to $H_p$-spaces}, Cambridge Tracts in
Mathematics, \textbf{115}, Cambridge University Press, 1998.

\bibitem{Leb1907rcmp}%
H.~Lebesgue, \textit{Sur le probl\`eme de Dirichlet}, Rend. Circ. Mat. di
Palermo, \textbf{29} (1907), 371--402.

\bibitem{Maz2009sbm}%
M.~Ya.~Mazalov, \textit{The Dirichlet problem for polyanalytic functions},
Sb.~Math., \textbf{200} (2009), no.~10, 1473--1493.

\bibitem{Maz2016spmj}%
M.~Ya.~Mazalov, \textit{An example of a non-rectifiable Nevanlinna contour},
St.~Petersburg Math.~J., \textbf{27} (2016), no.~4, 625--630.

\bibitem{Maz2018spmj}%
M.~Ya.~Mazalov, \textit{On Nevanlinna domains with fractal boundaries},
St.~Petersburg Math.~J., \textbf{29} (2018), no.~5, 777--791.

\bibitem{MazParFed2012rms}%
M.~Ya.~Mazalov, P.~V.~Paramonov, K.~Yu.~Fedorovskiy, \textit{Conditions for
$C^m$-approximabi\-lity of functions by solutions of elliptic equations},
Russ. Math. Surveys, \textbf{67} (2012), no.~6, 1023--1068.

\bibitem{Nar1973book}%
R.~Narasimhan, \textit{Analysis on real and complex manifolds}, Advanced
studies in pure mathematics, \textbf{1}, North--Holland Publishig Company,
Amsterdam, 1968.

\bibitem{Par1994sbm}%
P.~V.~Paramonov, \textit{$C^m$-approximations by harmonic polynomials on
compact sets in~$\mathbb R^n$}, Russian Acad. Sci. Sb. Math., \textbf{78}
(1994), no.~1, 231--251.

\bibitem{ParFed1999sbm}%
P.~V.~Paramonov, K.~Yu.~Fedorovskiy, \textit{Uniform and
$C^1$-approximability of functions on compact subsets of $\mathbb R^2$ by
solutions of second-order elliptic equations}, Sb.~Math., \textbf{190}
(1999), no.~2, 285--307.

\bibitem{ParVer1994msand}%
P.~V.~Paramonov, J.~Verdera,  \textit{Approximation by solutions of elliptic
equations on closed subsets of Euclidean space}, Math. Scand., \textbf{74}
(1994), no.~2, 249--259.

\bibitem{Pri1950book}%
I.~I.~Privalov, \textit{Boundary properties of analytic functions}, 2nd ed.,
GITTL, Moscow, 1950; German transl.: \textit{Randeigenschaften analytischer
Funktionen}, Hochschulb\"ucher f\"ur Mathematik, Bd.~25. VEB Deutscher Verlag
der Wissenschaften, Berlin, 1956.

\bibitem{Pom1975book}%
Ch.~Pommerenke, \textit{Univalent functions}, Studia
Mathematica\;/\;Mathematische Lehrb\"ucher, Vandenhoeck \& Ruppert,
G\"ottingen, 1975.

\bibitem{Pom1992book}%
Ch.~Pommerenke, \textit{Boundary behaviour of conformal maps}, Grundlehren
der mathematischen Wissenschaften, \textbf{299}, Springer-Verlag, 1992.

\bibitem{Rud1956pams}%
W.~Rudin, \textit{Boudary values of continuous analytic functions}, Proc.
Amer. Math. Soc., \textbf{7} (1956), no.~5, 808--811.

\bibitem{Tar1987sbm}%
N.~N.~Tarkhanov, \textit{Uniform approximation by solutions of elliptic
systems}, Math.~USSR-Sb., \textbf{61} (1988), no.~2, 351--377.	

\bibitem{Tre1961book}%
J.\,F.~Treves, \textit{Lectures on linear partial differential equations with
constant coefficients}, Notas de Matematica, Instituto de Matematica Pura e
Aplicada de Conselho Nacional de Pesquisas, Rio de Janeiro 1961.

\bibitem{VerVog1997tams}%
G.~C.~Verchota, A.~L.~Vogel, \textit{Nonsymmetric systems on nonsmooth planar
domains}, Trans. Amer. Math. Soc., \textbf{349} (1997), no.~11, 4501--4535.

\bibitem{Ver1987duke} %
J.~Verdera, \textit{$C^m$-approximation by solutions of elliptic equations,
and Calder\'on--Zygmund operators}, Duke Math.~J., \textbf{55} (1987), no.~1,
157--187.

\bibitem{Ver1993pjm} %
J.~Verdera, \textit{On the uniform approximation problem for the square of
the Cauchy--Riemann operator}, Pacific~J. Math., \textbf{159} (1993), no.~2,
379--396.

\bibitem{Ver1994nato} %
J.~Verdera, \textit{Removability, capacity and approximation}, Complex
potential theory, NATO Adv. Sci. Inst. Ser.~C Math. Phys. Sci., \textbf{439},
Dordrecht, Kluwer Acad. Publ., 1994, 419--473.

\bibitem{Vis1951sbm}%
M.~I.~Vishik, \textit{On strongly elliptic systems of differential
equations}, Mat.~Sb. (N.S.), \textbf{29(71)} (1951), no,~3, 615--676.

\bibitem{Wal1929bams}%
J.~L.~Walsh, \textit{The approximation of harmonic functions by harmonic
polynomials and by harmonic rational functions}, Bull. Amer. Math. Soc.,
\textbf{35} (1929), 499--544.

\bibitem{Zai2002mn}%
A.~B.~Zaitsev, \textit{Uniform approximability of functions by polynomials of
special classes on compact sets in $\mathbb R^2$}, Math. Notes, \textbf{71}
(2002), no.~1, 68--79.

\bibitem{Zai2004izv}%
A.~B.~Zaitsev, \textit{Uniform approximability of functions by polynomial
solutions of second-order elliptic equations on compact plane sets},
Izv.~Math., \textbf{68} (2004), no.~6, 1143--1156.

\bibitem{Zai2006psim}%
A.~B.~Zaitsev, \textit{Uniform approximation by polynomial solutions of
second-order elliptic equations, and the corresponding Dirichlet problem},
Proc. Steklov Inst. Math., \textbf{253} (2006), 57--70.

\end{thebibliography}
\end{document}